\documentclass[11pt,reqno]{amsart}
\usepackage[T1]{fontenc}
\usepackage{lmodern,amssymb,mathtools,mathrsfs,booktabs,indentfirst,tikz}
\usepackage[margin=1in]{geometry}
\usepackage{etoolbox}
\makeatletter
\@ifundefined{subjclassname@2020}{%
  \@namedef{subjclassname@2020}{\textup{2020} Mathematics Subject Classification}%
}{}
\patchcmd{\@settitle}{\uppercasenonmath\@title}{}{}{%
  \PackageError{manuscript}{Unable to preserve title capitalization}{}%
}
\patchcmd{\maketitle}{\uppercasenonmath\shorttitle}{}{}{%
  \PackageError{manuscript}{Unable to preserve running title capitalization}{}%
}
\patchcmd{\section}{\scshape}{\upshape}{}{%
  \PackageError{manuscript}{Unable to set sentence-case section headings}{}%
}
\patchcmd{\@setauthors}{\MakeUppercase{\authors}}{\authors}{}{%
  \PackageError{manuscript}{Unable to preserve author name capitalization}{}%
}
\patchcmd{\maketitle}{\@nx\MakeUppercase{\the\toks@}}{\the\toks@}{}{%
  \PackageError{manuscript}{Unable to preserve running author name capitalization}{}%
}
\makeatother
\usepackage[colorlinks=true,linkcolor=blue,citecolor=blue,urlcolor=blue]{hyperref}
\allowdisplaybreaks[2]
\newtheorem{theorem}{Theorem}[section]
\newtheorem{lemma}[theorem]{Lemma}
\newtheorem{proposition}[theorem]{Proposition}
\newtheorem{corollary}[theorem]{Corollary}
\theoremstyle{remark}

\newcommand{\R}{\mathbb R}

\newcommand{\dd}{\,\mathrm{d}}
\DeclareMathOperator{\tr}{tr}
\DeclareMathOperator{\adj}{adj}
\DeclareMathOperator{\spanop}{span}
\numberwithin{equation}{section}
\title[Neumann eigenvalue means on triangles]{Sharp Neumann eigenvalue means on triangles}
\author{Guowei Dai}
\address[Guowei Dai]{School of Mathematical Sciences, Dalian University of Technology,
Dalian 116024, P.R. China}
\email{daiguowei@dlut.edu.cn}
\author{Yingxin Sun}
\address[Yingxin Sun]{School of Mathematical Sciences, Dalian University of Technology,
Dalian 116024, P.R. China}
\email{sunyingxin2023@mail.dlut.edu.cn}
\date{}
\begin{document}
\begin{abstract}
We prove two conjectures of Laugesen and Siudeja: the equilateral triangle
uniquely maximizes the harmonic mean of the first two positive Neumann
eigenvalues at fixed perimeter and their geometric mean at fixed area.
The proof uses the two equilateral eigenfunctions and four projected
polynomial corrections. Triangle symmetry reduces the associated
variational matrices to two-by-two operations and two scalar invariants.
Angular comparison and horizontal trial functions then yield the global
inequalities. The same argument gives explicit quadratic stability bounds
in terms of normalized side-length asymmetry.
\end{abstract}
\maketitle

\section{Introduction}

The low Neumann spectrum describes the nonconstant modes of a membrane
with free boundary and the decay rates of heat flow with an insulating
boundary. For a nondegenerate planar triangle $T$, we write
\[
 0=\mu_1\left(T\right)<\mu_2\left(T\right)\leq\mu_3\left(T\right)\leq\cdots,
\]
counting eigenvalues with multiplicity. They are the eigenvalues of
$-\Delta$ with Neumann boundary condition: an eigenfunction $f$ satisfies
$-\Delta f=\mu f$ in $T$ and $\partial_n f=0$ on its boundary, where
$\Delta$ is the planar Laplacian and $n$ is the outward unit normal.
At the vertices the boundary condition is understood through the weak
formulation recalled in Section~\ref{sec:preliminaries}.
Let $A$ and $L$ denote the area and perimeter of $T$. A dilation by a
factor $t>0$ multiplies its eigenvalues by $t^{-2}$, so products with
$A$ or $L^2$ are unchanged by scaling. For positive numbers $x,y$, set
\[
 H\left(x,y\right)=\frac{2xy}{x+y},\qquad G\left(x,y\right)=\sqrt{xy},
\]
the harmonic and geometric means, respectively.

Shape optimization asks which domains maximize or minimize a spectral
quantity under a prescribed geometric constraint. The classical
Szeg\H{o}--Weinberger theorem identifies the disk as the maximizer of
the first positive Neumann eigenvalue at fixed area
\cite{Szego1954,Weinberger1956}. In the class of triangles,
Laugesen and Siudeja proved that the equilateral triangle maximizes
$\mu_2 A$ and $\mu_2L^2$ \cite[Theorem~3.1]{LS2009}. They also proved
the sharp two-eigenvalue estimate
\[
 H\left(\mu_2,\mu_3\right)A\leq\frac{4\pi^2}{3\sqrt3}
\]
with equality only for the equilateral triangle
\cite[Theorem~3.4]{LS2009}.

The next question is whether one can strengthen the geometric
normalization or the spectral mean in this last inequality. Since
$L^2\geq12\sqrt3 A$ for triangles \cite[Lemma~3.2]{LS2009}, replacing
$A$ by $L^2/\left(12\sqrt3\right)$ gives a stronger assertion. Replacing $H$ by
$G\geq H$ also gives a stronger assertion. Laugesen and Siudeja proposed
these two bounds in Conjectures~9.5 and~9.6 of \cite{LS2009}; they were
recorded together as Conjecture~6.34 in \cite[p.~175]{LS2017}.
The 2009 paper numbers the zero eigenvalue by $0$, whereas we number
it by $1$. The historical record also includes an announcement of
the geometric-mean bound in the 2009 Oberwolfach report, attributed
there to unpublished Mathematica-assisted work of Siudeja
\cite[pp.~405--406]{OW2009}. Enache and Philippin obtained related
product bounds for polygons \cite{EP2013}; the 2017 chapter discusses
them as partial progress toward the geometric-mean assertion
\cite[p.~175]{LS2017}.

We prove both inequalities with their sharp constants and equality cases.

\begin{theorem}\label{thm:main-algebraic}
For every nondegenerate planar triangle $T$,
\[
 H\left(\mu_2,\mu_3\right)L^2\leq16\pi^2,\qquad
 G\left(\mu_2,\mu_3\right)A\leq\frac{4\pi^2}{3\sqrt3}.
\]
Equality in either inequality holds if and only if $T$ is equilateral.
\end{theorem}

The proof also measures how the spectral quantities decrease away from
the maximizing shape. If $\ell_1,\ell_2,\ell_3$ are the side lengths,
define their normalized asymmetry by
\[
 \mathcal A\left(T\right)=
 \frac{\sum_{i<j}\left(\ell_i-\ell_j\right)^2}{L^2}.
\]
This scale-invariant quantity is nonnegative and vanishes precisely
when $T$ is equilateral.

\begin{corollary}\label{cor:stability-proof}
For every nondegenerate planar triangle,
\[
 1-\frac{H\left(\mu_2,\mu_3\right)L^2}{16\pi^2}
 \geq\frac{\mathcal A\left(T\right)}{4800},\qquad
 1-\frac{27\mu_2\mu_3A^2}{16\pi^4}
 \geq\frac{2}{1251}\mathcal A\left(T\right).
\]
\end{corollary}

These bounds give quantitative rigidity: a small spectral deficit
forces the relative side-length differences to be small. We do not
determine the optimal constants in these stability inequalities.

There are two distinct obstacles to a sharp proof. First, normalizing
the longest side does not prevent the height from tending to zero.
An estimate restricted to a compact set of nondegenerate shapes cannot
cover all triangles. Second, equality occurs at a double eigenvalue.
For a sharp two-function Ritz comparison, the trial space at the
equilateral triangle must reproduce its full first positive eigenspace.
Transporting that eigenspace by an affine map, a linear transformation
followed by a translation, preserves equality, but
its product estimate exceeds the conjectured bound on every other
shape; the exact excess is computed in
\eqref{eq:affine-obstruction}. Its harmonic estimate has the same
second-order obstruction near equality. Thus the known area bound
and the basic affine trial space do not by themselves supply the
required perimeter and product comparisons.

Our starting point is the trial-function method of
\cite[Sections~4 and~7]{LS2009}. Symmetry and affine transplantation
also underlie the eigenvalue-sum estimates of Laugesen and Siudeja
\cite{LS2011}, where averaging over rotations produces sharp
normalizations involving the moment of inertia, the integral of squared
distance from the centroid. Here the geometric constraints are area
and perimeter, and the spectral quantities are two nonlinear means.
The additional ingredient is a fixed six-dimensional space consisting
of the two equilateral modes and four projections of polynomial
directions, orthogonal to the modes in the mass inner product.
For each affine metric, we minimize the energy over the
four complementary directions while retaining their contribution to
the mass. This produces a two-dimensional trial space whose energy
matrix is a Schur complement, the matrix remaining after the
complementary variables have been eliminated. The correction lowers
energy and increases mass at the quadratic order where the affine
trial estimate fails. The worked coefficient calculation in
Appendix~\ref{app:integrals} makes this mechanism explicit.

The six symmetries of the equilateral triangle impose a further
structure. They reduce the variational determinants to two scalar
invariants, quantities unchanged by those symmetries, of the part
of the affine metric with zero trace. An angular
comparison then makes the geometric-mean target smallest on one
isosceles direction. For the harmonic mean, the perimeter itself
varies with angle, and its derivative supplies the additional
comparison needed to reach an isosceles direction. The shapes with
larger metric anisotropy are controlled by horizontal quartic trial
functions and an elementary convexity estimate. The resulting proof
uses explicit two-by-two matrix operations, one-variable inequalities,
and a fixed set of rational bounds on scalar integrals. It does not
require the eigenfunctions of a general triangle.

The wider polygonal and convex-domain problems illustrate what remains
outside these results. The perimeter bound for the first positive
eigenvalue in the convex planar class is formulated in
\cite[Conjecture~1.1]{LenaRohleder2024}, with the square and equilateral
triangle as equality candidates. It is known for domains with two axes
of symmetry \cite[Theorem~1.2]{HenrotLemenantLucardesi2024} and for
parallelograms \cite[Theorem~3.7]{LenaRohleder2024}. The existence and
structure of convex optimizers under geometric constraints are studied
in \cite{BogoselHenrotMichetti2024}. For the area constraint, Endo and
Osting have recently proved the square's optimality among convex
quadrilaterals \cite{EndoOsting2026}. These results concern different
geometric classes or different spectral quantities. Extending sharp
mean inequalities to such classes, treating higher eigenvalues, and
finding the best stability constants require information not supplied
by the triangular symmetry reduction developed here.

Section~\ref{sec:preliminaries} gives the variational formulas and the
shape normalization. Section~\ref{sec:horizontal} treats degeneration
and the large-anisotropy region. The symmetric trial space and its
matrix formulas are constructed in Section~\ref{sec:trial-space}.
Sections~\ref{sec:geometric-short} and~\ref{sec:harmonic-short} prove
the two comparisons and their quantitative estimates, and
Section~\ref{sec:completion} assembles the results.
Appendix~\ref{app:integrals} supplies the scalar integration formulas
and worked calculations underlying the fixed coefficient bounds.

\section{Variational comparison and normalized geometry}\label{sec:preliminaries}

We first record the variational facts that turn trial functions into
bounds for both means. The real Sobolev space $H^1\left(T\right)$ consists of
square-integrable functions with square-integrable weak first derivatives.
Its mean-zero subspace is
\[
 H^1_*\left(T\right)=\left\{f\in H^1\left(T\right):\int_Tf\dd\boldsymbol{x}=0\right\},
\]
where $\boldsymbol{x}=\left(x_1,x_2\right)$ and $\dd\boldsymbol{x}$ denotes planar
Lebesgue measure. The weak Neumann equation is
\[
 \int_T\nabla f\cdot\nabla g\dd\boldsymbol{x}
 =\mu\int_Tfg\dd\boldsymbol{x}\qquad\left(g\in H^1\left(T\right)\right).
\]
Here $\nabla$ is the gradient. For $f\ne0$ its Rayleigh quotient is
$\mathscr R_T\left(f\right)=\int_T\left|\nabla f\right|^2\dd\boldsymbol{x}/
\int_Tf^2\dd\boldsymbol{x}$. The min--max principle gives
\[
 \mu_{j+1}\left(T\right)=
 \min_{\substack{E\subset H^1_*\left(T\right)\\\dim E=j}}
 \max_{0\ne f\in E}\mathscr R_T\left(f\right),\qquad j\geq1
\]
\cite[Section~2]{LenaRohleder2024}.

For matrices, superscript $\mathsf T$ denotes transpose, $I_n$ the
$n\times n$ identity, $\det$ the determinant, and $\tr$ the trace.
The adjugate $\adj B$ is the transpose of the cofactor matrix;
for invertible $B$, it equals $\left(\det B\right)B^{-1}$.
We write $B\preceq C$ when $x^{\mathsf T}\left(C-B\right)x\geq0$ for every
real vector $x$, and $B\succ0$ when its quadratic form is strictly
positive for $x\ne0$. Identity matrices without a subscript have the
size determined by the formula. The matrix norms used below are
\[
 \left\|B\right\|_F^2=\sum_{i,j}\left|B_{ij}\right|^2,
 \qquad \left\|B\right\|_{\mathrm{op}}=\sup_{\left|x\right|=1}\left|Bx\right|,
\]
where $\left|x\right|$ is the Euclidean length. Matrix identities and positive
definite order are used in the conventions of
\cite[Chapters~0 and~7]{HornJohnson2013}.

Similarity invariance now permits a normalization of the shape.
By a similarity and reflection put a longest side at $\left(-1,0\right),\left(1,0\right)$
and the remaining vertex at $\left(a,b\right)$, where $a\geq0,b>0$.
If the other sides are $r+s,r-s$, then
\begin{equation}\label{eq:shape-proof}
\begin{gathered}
 1<r\leq2,\qquad 0\leq s\leq2-r,\qquad a=rs,\\
 d=b^2=\left(r^2-1\right)\left(1-s^2\right),\qquad
 L=2\left(r+1\right),\qquad A=b.
\end{gathered}
\end{equation}
Indeed, the triangle inequality gives $r>1$, the longest-side condition
gives $r+s\leq2$, and subtraction of the squared other side lengths gives
$a=rs$. Substitution gives $d$ and the converse as well.
Also
$a^2+d=r^2+s^2-1$.
The equilateral shape is $\left(r,s\right)=\left(2,0\right)$; see Figure~\ref{fig:normalization}.

\begin{figure}[ht]
\centering
\begin{tikzpicture}[scale=1.25]
 \coordinate (P) at (-1.6,0);
 \coordinate (Q) at (1.6,0);
 \coordinate (R) at (.56,1.84);
 \coordinate (H) at (.56,0);
 \fill[blue!5] (P)--(Q)--(R)--cycle;
 \draw (P)--node[above left] {$r+s$} (R)--node[above right] {$r-s$} (Q)--cycle;
 \draw[dashed] (R)--node[right] {$b$} (H);
 \draw[<->] (-1.6,-.30)--node[below] {$2$} (1.6,-.30);
 \fill (P) circle (1pt) (Q) circle (1pt) (R) circle (1pt);
 \node[below left] at (P) {$\left(-1,0\right)$};
 \node[below right] at (Q) {$\left(1,0\right)$};
 \node[above] at (R) {$\left(a,b\right)$};
 \node[below right] at (H) {$\left(a,0\right)$};
\end{tikzpicture}
\caption{Normalization by a longest side. The area is $b$ and the
perimeter is $2\left(r+1\right)$.}\label{fig:normalization}
\end{figure}

The asymmetry defined in the introduction satisfies
\[
 \mathcal A\left(T\right)=\frac{2\left(2-r\right)^2+6s^2}{4\left(r+1\right)^2}\leq\frac12.
\]
The last bound also follows from $\sum_i\ell_i^2\leq L^2/2$,
which is a consequence of the triangle inequality.
We shall use the metric anisotropy parameter
\begin{equation}\label{eq:eta-definition}
 \eta^2=1-\frac{12\left(r^2-1\right)\left(1-s^2\right)}{\left(r^2+s^2+2\right)^2},
 \qquad 0\leq\eta<1.
\end{equation}
Its interpretation as the anisotropy of the affine energy metric
will be given in Section~\ref{sec:trial-space}. The two cutoffs used
below are $r_0=47/40$ and $\eta_0=1397/1803$.
The variables $r,s$ describe the side geometry directly and are useful
for horizontal trial functions. The parameter $\eta$ instead measures
the departure of the affine energy metric from an isotropic one.
These two descriptions will let us choose the comparison best suited
to each part of the shape domain.

A two-dimensional trial space controls both positive eigenvalues in
a single variational comparison. This is particularly useful at the
equilateral triangle, where the two eigenvalues coincide. The following reduction expresses both means directly
through the mass and energy of any basis of that space. In particular,
the basis need not be orthogonal, and the trial functions need not
satisfy a boundary condition: membership in $H^1_*\left(T\right)$ is
sufficient for the Neumann variational principle.

\begin{lemma}[Two-function Ritz reduction]\label{lem:ritz-proof}
Let $f_1,f_2$ be linearly independent real functions in
$H^1_*\left(T\right)$, and define their mass and energy matrices by
\[
 M_{ij}=\int_Tf_if_j\dd\boldsymbol{x},\qquad
 K_{ij}=\int_T\nabla f_i\cdot\nabla f_j\dd\boldsymbol{x},
 \qquad i,j\in\left\{1,2\right\}.
\]
Both matrices are positive definite. The generalized eigenvalues,
called the Ritz values, are the numbers $\nu$ for which
$Kc=\nu Mc$ has a solution $c\in\R^2\setminus\left\{0\right\}$.
Their ordered values satisfy
\[
 0<\nu_1\leq\nu_2,\qquad
 \mu_2\leq\nu_1,\qquad \mu_3\leq\nu_2,
\]
and
\[
 H\left(\nu_1,\nu_2\right)
 =\frac{2\det K}{\tr\left(\adj\left(K\right)M\right)},
 \qquad
 \nu_1\nu_2=\frac{\det K}{\det M}.
\]
\end{lemma}
\begin{proof}
Set $V=\spanop\left\{f_1,f_2\right\}$. A coefficient vector
$c=\left(c_1,c_2\right)^{\mathsf T}$ represents the function
$f_c=c_1f_1+c_2f_2$, and
\[
 \int_T f_c^2\dd\boldsymbol{x}=c^{\mathsf T}Mc,
 \qquad
 \int_T\left|\nabla f_c\right|^2\dd\boldsymbol{x}
 =c^{\mathsf T}Kc.
\]
The first expression is positive for $c\ne0$ by independence.
If the second vanished, $f_c$ would be constant on the connected
triangle. Its zero mean would then give $f_c=0$, contradicting
independence. Thus both quadratic forms are positive definite.

The finite-dimensional variational characterization of generalized
eigenvalues \cite[Chapter~7]{HornJohnson2013} gives
\[
 \nu_1=\min_{0\ne f\in V}\mathscr R_T\left(f\right),
 \qquad
 \nu_2=\max_{0\ne f\in V}\mathscr R_T\left(f\right).
\]
Restricting the minimization defining $\mu_2$ to $V$ gives
$\mu_2\leq\nu_1$. In the characterization of $\mu_3$, the space $V$
is one admissible two-dimensional competitor, so
$\mu_3\leq\nu_2$. These are the two applications of the Neumann
min--max principle needed here.

It remains to express the symmetric combinations without solving
for the Ritz values separately. Taking determinants in the
generalized characteristic equation gives
$\nu_1\nu_2=\det K/\det M$. For the harmonic mean it is more
convenient to use reciprocals. The eigenvalues of $K^{-1}M$ are
$\nu_1^{-1},\nu_2^{-1}$; hence
\[
 \frac1{\nu_1}+\frac1{\nu_2}
 =\tr\left(K^{-1}M\right)
 =\frac{\tr\left(\adj\left(K\right)M\right)}{\det K}.
\]
The definition of $H$ now gives the asserted formula.
\end{proof}

Both means are increasing in each positive argument. Consequently,
any such trial space supplies upper bounds for the two quantities
in Theorem~\ref{thm:main-algebraic}. The advantage of the lemma is
that these bounds depend only on determinants and one mixed trace.
We can therefore choose trial functions for their geometric and
symmetry properties, and compare the resulting matrices without
tracking the individual Ritz eigenvalues.

\section{Horizontal trial functions}\label{sec:horizontal}

The spectral comparison must remain uniform as the height decreases.
Trial functions depending only on the horizontal coordinate provide
this control. A strengthened bound from the same pair will also cover
the large-$\eta$ region of \eqref{eq:eta-definition}.

The reason for this choice is geometric. A horizontal function has no
vertical derivative, and the factor coming from the height cancels
between its energy and mass. Its normalized moments therefore depend
on the apex coordinate $a$, but not on the height $b$. This gives a
uniform estimate even for nearly degenerate triangles. The role of
this pair is to provide a strict bound in the regions treated below.
The symmetric construction in Section~\ref{sec:trial-space} will then
supply the sharp comparison near equality.

\begin{lemma}[A horizontal quartic cap]\label{lem:analytic-cap-horizontal}
For a normalized triangle satisfying \eqref{eq:shape-proof}, suppose
$1<r\leq47/40$.
Writing $A=b$, $L=2\left(r+1\right)$, the first two positive Neumann eigenvalues satisfy
\[
 H\left(\mu_2,\mu_3\right)L^2<\frac{2523}{16}<16\pi^2,
 \qquad\mu_2\mu_3A^2<\frac{27405}{800}<\frac{16\pi^4}{27}.
\]
The normalized deficits are respectively larger than
$661/1577536$ and $733/2560$.
\end{lemma}
\begin{proof}
Let $\mathbb E F=A^{-1}\int_TF\dd\boldsymbol{x}$ denote the uniform average
over $T$, and let $X_1$ be the horizontal coordinate. Choose
\begin{align*}
 f_1&=X_1-\frac27X_1^3-\mathbb E\left(X_1-\frac27X_1^3\right),\\
 f_2&=X_1^2-\frac59X_1^4-\mathbb E\left(X_1^2-\frac59X_1^4\right).
\end{align*}
They are mean zero and independent, by their distinct highest powers of $X_1$.
Define $M_{ij}=\mathbb E\left(f_if_j\right)$,
$K_{ij}=\mathbb E\left(\nabla f_i\cdot\nabla f_j\right)$, and
$\mathcal T=\operatorname{tr}\left(\operatorname{adj}\left(K\right)M\right)>0$.
The barycentric monomial formula in Appendix~\ref{app:integrals} yields
\[
 \mathbb E X_1^k=\frac{2}{\left(k+1\right)\left(k+2\right)}
 \sum_{\substack{0\leq\ell\leq k\\k-\ell\ \mathrm{even}}}a^\ell.
\]
Put $z=a^2\in\left[0,1\right]$. Direct substitution gives
\begin{align*}
175032900\left[25\mathcal T-6\left(1+\frac25z\right)\det K\right]
={}&255035+10609438z+\frac{12532783}{5}z^2
 -\frac{4059404}{5}z^3\\
 &-\frac{24204072}{5}z^4-324712z^5-36800z^6,\\
175032900\left(90\det M-\det K\right)
={}&1834965+1202485z+53477671z^2-43408547z^3\\
 &+7010208z^4+906756z^5-122200z^6+4550z^7.
\end{align*}
Use $z^k\leq z$ for the negative higher powers in the first identity
and $z^3,z^6\leq z^2$
in the second, and discard other positive terms. The right sides are bounded
below by
\[
 255035+\frac{22976154}{5}z+\frac{12532783}{5}z^2>0,\qquad
 1834965+1202485z+9946924z^2>0.
\]
Consequently the Ritz values obey
\begin{equation}
 H\left(\nu_1,\nu_2\right)=\frac{2\det K}{\mathcal T}
 <\frac{25}{3\left(1+2a^2/5\right)}\leq\frac{25}{3},\qquad
 \nu_1\nu_2=\frac{\det K}{\det M}<90.\label{eq:horizontal-ritz}
\end{equation}
Since $L\leq87/20$ and $A^2\leq609/1600$,
\[
 H\left(\nu_1,\nu_2\right)L^2<\frac{25}{3}\left(\frac{87}{20}\right)^2
 =\frac{2523}{16},\qquad
 \nu_1\nu_2A^2<\frac{27405}{800}.
\]
The bounds $\pi>157/50>3$ give
\[
16\pi^2>\frac{394384}{2500}>\frac{2523}{16},
\qquad\frac{16\pi^4}{27}>48>\frac{27405}{800}.
\]
Also
\[
1-\frac{2523/16}{394384/2500}=\frac{661}{1577536},\qquad
1-\frac{27405/800}{48}=\frac{733}{2560}.
\]
Apply Lemma~\ref{lem:ritz-proof} and the monotonicity of the means.
\end{proof}

The same horizontal product estimate $\nu_1\nu_2<90$ holds whenever
$0\leq a\leq1$. Thus, on the slightly larger cap $1<r\leq6/5$,
\[
 \nu_1\nu_2A^2<90\frac{11}{25}=\frac{198}{5}<48
 <\frac{16\pi^4}{27}.
\]
The corresponding normalized geometric deficit exceeds $7/40$.

\begin{lemma}[A uniform horizontal harmonic cap]\label{lem:high-horizontal}
Use the shape variables in \eqref{eq:shape-proof} and the anisotropy
parameter \eqref{eq:eta-definition}. Let $r_0=47/40$ and
$\eta_0=1397/1803$.
If either $r\leq r_0$ or $\eta\geq\eta_0$, then
\[
 H\left(\mu_2,\mu_3\right)L^2<\frac{2523}{16}<16\pi^2,
 \qquad
 1-\frac{H\left(\mu_2,\mu_3\right)L^2}{16\pi^2}
 >\frac{661}{1577536}.
\]
\end{lemma}
\begin{proof}
Write $H=H\left(\mu_2,\mu_3\right)$. The calculation in Lemma~\ref{lem:analytic-cap-horizontal}, in particular \eqref{eq:horizontal-ritz}, uses only $0\leq a\leq1$, so it
applies to every normalized triangle. If $r\leq r_0$, the estimate
$L\leq87/20$ immediately gives the stated bound, as above.

It remains to consider $r\geq r_0$, $\eta\geq\eta_0$. Write
$R=r^2$, $t=s^2$, and retain $z=Rt=a^2$. The expression for $\eta^2$
increases with $t$. Substitution at $t=\left(2-r\right)^2$ gives
\[
 \eta\leq\frac{2\left[1-\left(r-1\right)^2\right]}{2+\left(r-1\right)^2}.
\]
Thus
\begin{equation}
 r\leq1+\sqrt{\frac{2\left(1-\eta_0\right)}{2+\eta_0}}
 =1+\sqrt{\frac{812}{5003}}<\frac{281}{200}=:r_1.         \label{eq:eta-rmax}
\end{equation}
The final comparison is certified by
\[
 \left(\frac{81}{200}\right)^2-\frac{812}{5003}
 =\frac{344683}{200120000}>0.
\]

Set $u=\eta^2$ and
\[
 u_{\rm iso}\left(r\right)=\left(\frac{4-r^2}{r^2+2}\right)^2.
\]
On $r_0\leq r\leq r_1$, one has $1<R<2$ and
\[
 u_{\rm iso}''\left(r\right)=\frac{24\left(-3R^2+26R-8\right)}{\left(R+2\right)^4}>0.
\]
A convex function lies below its secant. At the left endpoint,
\[
 u_{\rm iso}\left(r_0\right)
 =\left(\frac{4191}{5409}\right)^2
 =\left(\frac{1397}{1803}\right)^2=\eta_0^2.
\]
The endpoint bounds
\[
 \eta_0^2>\frac35,\qquad
 u_{\rm iso}\left(r_1\right)=\left(\frac{81039}{158961}\right)^2
 <\left(\frac{51}{100}\right)^2
\]
give
\[
 \frac{\eta_0^2-u_{\rm iso}\left(r_1\right)}{r_1-r_0}
 >\frac{3/5-\left(51/100\right)^2}{23/100}=\frac{3399}{2300}>
 \frac{29}{20}.
\]
Combining this slope bound with the secant inequality gives, with
$\Delta=r-r_0\geq0$,
\begin{equation}
 \eta_0^2-u_{\rm iso}\left(r\right)\geq\frac{29}{20}\Delta.          \label{eq:eta-secant}
\end{equation}

At fixed $r$, the derivative with respect to $z=a^2$ is
\begin{equation}
 \frac{\partial u}{\partial z}
 =\frac{12\left(R-1\right)\left(R+4-t\right)}{R\left(R+t+2\right)^3}
 \leq\frac{12\left(R-1\right)\left(R+4\right)}{R\left(R+2\right)^3}<\frac{57}{100}.       \label{eq:eta-a-derivative}
\end{equation}
For the last inequality put $y=2-R$. By \eqref{eq:eta-rmax}, $0<y<2/3$. Clearing
positive denominators gives the polynomial
\begin{align*}
 &57R\left(R+2\right)^3-1200\left(R-1\right)\left(R+4\right)\\
 &\qquad=100\left(.96-7.2y+29.04y^2-7.98y^3+.57y^4\right)\\
 &\qquad\geq100\left(.96-7.2y+23.72y^2\right)>0.
\end{align*}
The final quadratic is positive because
$7.2^2<4\left(.96\right)\left(23.72\right)$.
Integrating \eqref{eq:eta-a-derivative} from $z=0$ and using $u\geq\eta_0^2$ with \eqref{eq:eta-secant}
now yields
\begin{equation}
 \frac{57}{100}z\geq u-u_{\rm iso}\left(r\right)
 \geq\frac{29}{20}\Delta,
 \qquad z\geq\frac{145}{57}\Delta\geq\frac52\Delta.       \label{eq:eta-a-lower}
\end{equation}

Put $C=r_0+1=87/40$. Combining \eqref{eq:horizontal-ritz} and \eqref{eq:eta-a-lower} gives
\[
 HL^2<\frac{100}{3}\frac{\left(C+\Delta\right)^2}{1+\Delta},
 \qquad0\leq\Delta<\frac{23}{100}.
\]
The derivative of $\left(C+\Delta\right)^2/\left(1+\Delta\right)$ has the sign of
$\Delta-\left(C-2\right)$, so its maximum on this interval is at an endpoint.
The larger endpoint is $\Delta=0$, since
\[
 C^2-\frac{\left(C+23/100\right)^2}{1+23/100}
 =\frac{5543}{196800}>0.
\]
Therefore
\[
 HL^2<\frac{100}{3}\left(\frac{87}{40}\right)^2
 =\frac{2523}{16}.
\]
Finally, $\pi>157/50$ gives
\[
 1-\frac{2523/16}{16\left(157/50\right)^2}
 =\frac{661}{1577536}>0.
\]
This proves both conclusions.
\end{proof}

\section{The symmetric trial space}\label{sec:trial-space}

The horizontal bounds leave a family of shapes containing the equality
case. We construct one reference space for these shapes and let the
affine metric determine its two-dimensional trial subspace.

\subsection*{The reference triangle and its first modes}
Let $D$ be the triangle with vertices
\[
 \left(-1/2,0\right),\qquad
 \left(1/2,0\right),\qquad \left(0,1/2\right).
\]
The coordinates $\left(x,\sqrt3z\right)$ identify it with an equilateral triangle
of side length one. Set $m=3/8$, $\lambda=16\pi^2/9$, and
\[
 \left\langle f,g\right\rangle=\int_Dfg\dd x\dd z,\qquad
 \mathcal E\left(f,g\right)=\int_D\left(f_xg_x+\frac13f_zg_z\right)\dd x\dd z.
\]
Subscripts denote partial derivatives. The first positive eigenspace
of the equilateral triangle has eigenvalue $\lambda$ and dimension two
\cite{McCartin2002}. We use the explicit basis
\begin{align*}
 u_1&=\sin\left(4\pi x/3\right)+2\sin\left(2\pi x/3\right)\cos\left(2\pi z\right),\\
 u_2&=\cos\left(4\pi x/3\right)-2\cos\left(2\pi x/3\right)\cos\left(2\pi z\right)
\end{align*}
from \cite[Section~4]{LS2009}. These functions have zero mean and satisfy
\[
 \left\langle u_i,u_j\right\rangle=m\delta_{ij},\qquad
 \mathcal E\left(u_i,f\right)=\lambda\left\langle u_i,f\right\rangle\quad\left(f\in H^1\left(D\right)\right),
\]
where $\delta_{ij}=1$ for $i=j$ and $0$ otherwise, and $\spanop$
denotes real linear span. Write
$U=\spanop\left\{u_1,u_2\right\}$, and define its orthogonal projection by
\[
 \Pi_Uf=\frac{\left\langle f,u_1\right\rangle}{m}u_1
       +\frac{\left\langle f,u_2\right\rangle}{m}u_2.
\]
After dilation to side length two the first positive eigenvalue is
$4\pi^2/9$.

The affine map
\[
 \left(x,z\right)\longmapsto \left(X_1,X_2\right)=\left(2x+2az,2bz\right)
\]
sends $D$ onto the normalized triangle and has Jacobian $4b$.
For a reference family $f_1,\ldots,f_n$, define its mass and derivative
Gram matrices by
\[
 M_{ij}=\int_Df_if_j\dd x\dd z,\quad
 \left(G_x\right)_{ij}=\int_D\left(f_i\right)_x\left(f_j\right)_x\dd x\dd z,
\]
\[
 \left(G_z\right)_{ij}=\int_D\left(f_i\right)_z\left(f_j\right)_z\dd x\dd z,\quad
 \left(G_{xz}\right)_{ij}=\int_D\left(f_i\right)_x\left(f_j\right)_z\dd x\dd z.
\]
If $F\left(X_1,X_2\right)=f\left(x,z\right)$, the chain rule gives
$F_{X_1}=f_x/2$ and $F_{X_2}=\left(f_z-af_x\right)/\left(2b\right)$.
The common Jacobian cancels from the generalized eigenvalue problem,
leaving the mass $M$ and energy $N/\left(4d\right)$, where
\begin{equation}\label{eq:energy-scaling-proof}
 N=\left(a^2+d\right)G_x+G_z-a\left(G_{xz}+G_{xz}^{\mathsf T}\right).
\end{equation}
This is the affine quadratic-form transformation used in
\cite[Section~2]{LenaRohleder2024}.

On the reference triangle put
\[
 x_c=x,\qquad y_c=\sqrt3\left(z-1/6\right),\qquad \mathcal R^2=x_c^2+y_c^2.
\]
Use the following polynomial functions. Their centering follows from
$\int_D\mathcal R^2\dd x\dd z=1/48$ and
$\int_Dy_c\left(3x_c^2-y_c^2\right)\dd x\dd z=-\sqrt3/360$, together with the
symmetry rules below:
\begin{align*}
 \phi_A&=\mathcal R^2-\frac1{12}
  +\frac34\left[y_c\left(3x_c^2-y_c^2\right)+\frac{\sqrt3}{90}\right],\\
 \phi_B&=x_c\left(x_c^2-3y_c^2\right),\\
 \phi_1^0&=x_c-x_cy_c-\frac53x_c\mathcal R^2,\\
 \phi_2^0&=y_c-\frac12\left(x_c^2-y_c^2\right)-\frac53y_c\mathcal R^2.
\end{align*}
These directions are chosen according to their symmetry. The function
$\phi_A$ is unchanged by all triangle symmetries, whereas $\phi_B$
is unchanged by rotations and changes sign under reflections.
The pair $\left(\phi_1^0,\phi_2^0\right)$ transforms as the pair of
equilateral modes. These transformation rules restrict which entries
can occur in the variational matrices. Projection will make all four
directions mass-orthogonal to $U$. The coefficients in the displayed
polynomials give a positive quadratic gain in both target inequalities,
as calculated in Appendix~\ref{app:integrals}; no optimality of this
choice is asserted.

\paragraph{The symmetry action and the projection.}
In these centered coordinates, let
\[
 \varrho\binom{x_c}{y_c}=
 \begin{pmatrix}-1/2&-\sqrt3/2\\ \sqrt3/2&-1/2\end{pmatrix}\binom{x_c}{y_c},
 \qquad \sigma\left(x_c,y_c\right)=\left(-x_c,y_c\right),\qquad \omega=e^{2\pi i/3}.
\]
These transformations generate the six-element dihedral group $D_3$,
the group of rotations and reflections preserving the equilateral
triangle. Here $i^2=-1$, and the complex notation encodes a real vector pair.
For a function $f$, its action is the pullback $f\mapsto f\circ g$.
If $Z_c=x_c+iy_c$, then
\[
 \phi_B=\Re Z_c^3,\qquad
 y_c\left(3x_c^2-y_c^2\right)=\Im Z_c^3,\qquad
 \phi_1^0+i\phi_2^0=Z_c-\frac i2\overline{Z_c}^{2}
                    -\frac53Z_c\left|Z_c\right|^2.
\]
Substitution gives the following complete generator rules:
\[
\begin{array}{c|cc}
 &\varrho&\sigma\\\hline
 \phi_A&\phi_A&\phi_A\\
 \phi_B&\phi_B&-\phi_B\\
 \phi_1^0+i\phi_2^0&\omega\left(\phi_1^0+i\phi_2^0\right)&-\overline{\phi_1^0+i\phi_2^0}\\
 u_1+iu_2&\omega\left(u_1+iu_2\right)&-\overline{u_1+iu_2}
\end{array}
\]
The last row follows directly from the displayed trigonometric modes
and the angle-addition identities. Thus the scalar and vector
transformation laws are explicit.
The $L^2$ inner product is preserved by both generators. Averaging over
the three rotations proves that $\phi_A,\phi_B$ are orthogonal to $U$.
The $2\times2$ matrix of inner products between the $\phi^0$ and $u$ pairs
commutes with the rotation and reflection matrices. Commutation with
the rotation gives a linear combination of $I$ and
$\left(\begin{smallmatrix}0&-1\\1&0\end{smallmatrix}\right)$;
commutation with the reflection removes the second term. Hence
\[
 \left\langle \phi_i^0,u_j\right\rangle=m\alpha\delta_{ij},\qquad
 \alpha=\frac{\left\langle \phi_1^0,u_1\right\rangle}{m},
\]
which proves the stated form of the $U$ projection.
The corrected pair $\left(\phi_1,\phi_2\right)$ obeys the same generator rules.

The projection scalar is

\[
 \phi_i=\phi_i^0-\alpha u_i,
\quad
 \alpha=\frac{9/4-27\sqrt3/16}{\pi^3}
       +\frac{729/64-405\sqrt3/32}{\pi^4}
       +\frac{10935}{128\pi^5}.
\]
Define the complementary space and its orthogonal direct sum with $U$ by
$V=\spanop\left\{\phi_A,\phi_B,\phi_1,\phi_2\right\}$ and $\mathcal W=U\oplus V$.

The cubic homogeneous parts of $\phi_A,\phi_B,\phi_1^0,\phi_2^0$ are independent:
the determinants in the $\left(x_c^3,x_cy_c^2\right)$ and $\left(x_c^2y_c,y_c^3\right)$ sectors are
$-20/3$ and $-5$, respectively. A nonzero combination of the two
trigonometric modes cannot be a polynomial. Indeed, continuity extends
any such identity to the edge $z=0$. Both restrictions are real analytic,
so equality on that interval implies equality on the whole real line.
The polynomial would then be periodic and hence constant. The displayed
odd and even trigonometric modes force both coefficients to vanish.
Projection therefore
preserves the four-dimensional independence.

\paragraph{The scalar integrals and their signs.}
Besides the mass inner product and $\mathcal E$, define the two
symmetric anisotropic forms
\begin{align*}
 \mathcal B_1\left(f,g\right)&=\int_D\left(f_xg_x-\frac13f_zg_z\right)\dd x\dd z,\\
 \mathcal B_2\left(f,g\right)&=\frac1{\sqrt3}\int_D\left(f_xg_z+f_zg_x\right)\dd x\dd z.
\end{align*}
For real metric coefficients $\tau,h,k$, the derivative energy for the metric
$\tau I+\left(\begin{smallmatrix}h&k\\k&-h\end{smallmatrix}\right)$
is $\tau\mathcal E+h\mathcal B_1+k\mathcal B_2$. In particular the
sign of every mixed entry in the stated matrix is fixed by these
definitions, together with $k=-\sqrt3a$.
The mass and energy scalars mean
\begin{align*}
 m_A&=\left\langle\phi_A,\phi_A\right\rangle,
 & e_A&=\mathcal E\left(\phi_A,\phi_A\right),\\
 m_B&=\left\langle\phi_B,\phi_B\right\rangle,
 & e_B&=\mathcal E\left(\phi_B,\phi_B\right),\\
 m_E&=\left\langle\phi_1,\phi_1\right\rangle
     =\left\langle\phi_2,\phi_2\right\rangle,
 & e_E&=\mathcal E\left(\phi_1,\phi_1\right)
     =\mathcal E\left(\phi_2,\phi_2\right).
\end{align*}
The remaining anisotropic scalars are
\[
\begin{array}{c|c@{\qquad}c|c}
 b_U&\mathcal B_1\left(u_1,u_1\right)&b_E&\mathcal B_1\left(\phi_1,\phi_1\right)\\
 g_A&\mathcal B_1\left(u_2,\phi_A\right)&g_B&\mathcal B_1\left(u_1,\phi_B\right)\\
 g_E&\mathcal B_1\left(u_1,\phi_1\right)&g_{AE}&\mathcal B_1\left(\phi_A,\phi_2\right)\\
 g_{BE}&\mathcal B_1\left(\phi_B,\phi_1\right)&t_0&\mathcal B_1\left(u_1,\phi_1^0\right)
\end{array}
\]
For the unprojected $\phi^0$ pair the raw integrals are
\[
 \left\langle \phi_1^0,\phi_1^0\right\rangle=
 \frac{7403}{1088640}+\frac{43\sqrt3}{45360},\quad
 \mathcal E\left(\phi_1^0,\phi_1^0\right)=\frac{221}{1296}-\frac{\sqrt3}{108},
\]
\[
 \mathcal B_1\left(\phi_1^0,\phi_1^0\right)=\frac{23}{162},\quad
 \mathcal B_1\left(\phi_A,\phi_2^0\right)=-\frac5{96}+\frac{67\sqrt3}{4320},\quad
 \mathcal B_1\left(\phi_B,\phi_1^0\right)=-\frac1{72}+\frac{\sqrt3}{120}.
\]
The correction formulas for $m_E,e_E,b_E,g_E,g_{AE},g_{BE}$ now
follow by expanding $\phi_i=\phi_i^0-\alpha u_i$. For $e_E$, use the weak
eigenfunction identity $\mathcal E\left(u_i,f\right)=\lambda\left\langle u_i,f\right\rangle$.
This explains the projection terms in every displayed scalar formula;
all signs and coefficients are determined by the stated integrals.

Evaluation of these integrals gives the following constants.  Let
$m=3/8$, $\lambda=16\pi^2/9$, and set
\begin{align*}
 t_0&=\frac{1215}{16\pi^3}-\frac3{8\pi}
                        -\frac{135\sqrt3}{16\pi^2},\\
 g_A&=-\frac{243}{16\pi^2}
      +\frac{9\sqrt3/4-81/32}{\pi}
      +\frac{6561}{256\pi^3},\qquad
 g_B=-\frac{27}{8\pi}+\frac{2187}{64\pi^3},\\
 m_A&=\frac{187}{153600}-\frac{\sqrt3}{2240},
 &e_A&=\frac{187}{1920}-\frac{\sqrt3}{40},\\
 m_B&=\frac1{13440},&e_B&=\frac1{40},\\
 m_E&=\frac{7403}{1088640}+\frac{43\sqrt3}{45360}-m\alpha^2,
 &e_E&=\frac{221}{1296}-\frac{\sqrt3}{108}-m\lambda\alpha^2,\\
 b_U&=\frac{81}{16},
 &b_E&=\frac{23}{162}-2\alpha t_0+\alpha^2b_U,\\
 g_E&=t_0-\alpha b_U,
 &g_{AE}&=-\frac5{96}+\frac{67\sqrt3}{4320}-\alpha g_A,\\
 &&g_{BE}&=-\frac1{72}+\frac{\sqrt3}{120}-\alpha g_B.
\end{align*}
The mass matrix in the order $\left(u_1,u_2,\phi_A,\phi_B,\phi_1,\phi_2\right)$ is
\[
 M=\operatorname{diag}\left(m,m,m_A,m_B,m_E,m_E\right).
\]
Here $\operatorname{diag}$ lists the diagonal entries; its complementary
block is denoted $M_V$. All six diagonal entries are positive;
this also follows directly from the independence of the stated
polynomials and their $U$ projections.

Write $q=a^2+d=r^2+s^2-1$ and introduce
\[
 \tau=\frac{q+3}{2},\qquad h=\frac{q-3}{2},\qquad k=-\sqrt3a,
 \qquad \mathsf T\left(h,k\right)=\begin{pmatrix}h&k\\k&-h\end{pmatrix}.
\]
The affine gradient metric in equilateral coordinates is
\[
 \begin{pmatrix}q&-\sqrt3a\\-\sqrt3a&3\end{pmatrix}
 =\tau I+\mathsf T\left(h,k\right),\qquad \tau^2-h^2-k^2=3d.
\]
Consequently the parameter \eqref{eq:eta-definition} is
$\eta=\sqrt{h^2+k^2}/\tau$. Define
\[
 u=\eta^2,\qquad h+ik=\sqrt{h^2+k^2}\,e^{i\theta},\qquad
 c=\cos\left(3\theta\right),\qquad
 v=\frac{h^3-3hk^2}{\tau^3}=\eta^3c.
\]
For $\eta>0$, the angle $\theta$ is taken modulo $2\pi$. At
$\eta=0$ it may be chosen arbitrarily: $v$ and every shape quantity
below are independent of this choice.
The normalized metric has eigenvalues $1\pm\eta$. Thus its positivity
is exactly $\eta<1$, and $\left|v\right|\leq u^{3/2}$. The case $\eta=0$ is
precisely the equilateral triangle.
By \eqref{eq:energy-scaling-proof}, the reference energy matrix for
the ordered basis $\left(u_1,u_2,\phi_A,\phi_B,\phi_1,\phi_2\right)$ is
\[
\begin{pmatrix}
 m\lambda\tau+b_Uh&b_Uk&g_Ak&g_Bh&g_Eh&g_Ek\\
 b_Uk&m\lambda\tau-b_Uh&g_Ah&-g_Bk&g_Ek&-g_Eh\\
 g_Ak&g_Ah&e_A\tau&0&g_{AE}k&g_{AE}h\\
 g_Bh&-g_Bk&0&e_B\tau&g_{BE}h&-g_{BE}k\\
 g_Eh&g_Ek&g_{AE}k&g_{BE}h&e_E\tau+b_Eh&b_Ek\\
 g_Ek&-g_Eh&g_{AE}h&-g_{BE}k&b_Ek&e_E\tau-b_Eh
\end{pmatrix}.
\]
Polynomial moments and the trigonometric integrals in
Appendix~\ref{app:integrals} give these formulas.
The generator rules determine the zero blocks and repeated entries;
the scalar integral definitions determine their values.
The positive gradient metric and the independence of these mean-zero
functions give $N\succ0$. Its complementary principal block is therefore
positive definite. The isotropic energies $e_A,e_B,e_E$ are positive
as well, so all inverses and normalizing square roots below are defined.

The first two coordinates alone explain the need for the complementary
directions. Put $e=m\lambda$ and $\beta=b_U/e=243/\left(32\pi^2\right)<1$.
Denote the first two-coordinate block of $N$ by $A_U$.
The affine image of $U$ has energy $A_U/\left(4d\right)$ and mass $mI_2$.
If $\nu_1,\nu_2$ denote its Ritz values in this paragraph, then
\begin{equation}\label{eq:affine-obstruction}
 \frac{27\nu_1\nu_2A^2}{16\pi^4}
 =\frac{1-\beta^2u}{1-u}
 =1+\frac{\left(1-\beta^2\right)u}{1-u}>1\qquad\left(u>0\right).
\end{equation}
The loss is quadratic in $\eta$ at equality. The corresponding harmonic
loss and the correction gain are calculated in Appendix~\ref{app:integrals}.

Partition $N$ at the first two coordinates and define the graph map by
\[
 N=\begin{pmatrix}A_U&B_U\\B_U^{\mathsf T}&C\end{pmatrix},\qquad
 \Lambda=-C^{-1}B_U^{\mathsf T}.
\]
We use the graph subspace $\left\{\left(x,\Lambda x\right):x\in\R^2\right\}$.
Its two reduced matrices are
\[
 S=A_U-B_UC^{-1}B_U^{\mathsf T},\qquad
 R=mI_2+\Lambda^{\mathsf T}M_V\Lambda.
\]
Write $\left(v_1,v_2,v_3,v_4\right)=\left(\phi_A,\phi_B,\phi_1,\phi_2\right)$.
The actual trial functions are
\[
 \psi_i=u_i+\sum_{\alpha=1}^4\Lambda_{\alpha i}v_\alpha,
 \qquad i=1,2.
\]
Their $U$ projections are $u_1,u_2$,
so the pair is independent and mean zero. The mass cross terms vanish
because $U\perp V$. Completing the energy square gives
\[
 \begin{pmatrix}x\\y\end{pmatrix}^{\mathsf T}N
 \begin{pmatrix}x\\y\end{pmatrix}
 =x^{\mathsf T}Sx+
 \left(y+C^{-1}B_U^{\mathsf T}x\right)^{\mathsf T}C
 \left(y+C^{-1}B_U^{\mathsf T}x\right).
\]
Thus the graph minimizes energy for every fixed $U$ coefficient vector
$x$, and its physical mass and energy are $R$ and $S/\left(4d\right)$.
This is the Schur-complement construction
\cite[Chapters~0 and~7]{HornJohnson2013}.

The construction corrects every element of $U$ simultaneously, giving
one two-dimensional space to which Lemma~\ref{lem:ritz-proof} applies.
The minimization holds with the $U$ coefficients fixed; the Rayleigh
quotient is then evaluated with both reduced matrices. Orthogonality
makes the two effects transparent:
\[
 S\preceq A_U,\qquad R\succeq mI_2.
\]
The correction lowers the energy and adds a nonnegative contribution
to the mass. Both changes help the desired upper bounds, and keeping
$R$ retains the second improvement exactly.

At the equilateral metric, $B_U=0$, so the correction vanishes and
the equality space is preserved. After normalizing by $\tau$, the
coupling is linear in $h/\tau,k/\tau$, while the inverse complementary
energy block remains bounded near equality. The energy reduction
and added mass therefore begin at quadratic order in $\eta$.
This is the order of the excess in \eqref{eq:affine-obstruction};
the leading-coefficient calculations in Appendix~\ref{app:integrals}
show that the gain is sufficient for both means at this order.

Introduce the scalar shift polynomials
\[
 \mathcal D\left(\xi\right)=\det\left(N+\xi M\right),\qquad
 \mathcal J\left(\xi\right)=\det\left(C+\xi M_V\right),
\]
where $\xi$ is an auxiliary real variable near zero. Sylvester's
determinant identity $\det\left(I+AB\right)=\det\left(I+BA\right)$
\cite[Chapter~0]{HornJohnson2013} gives
\[
 \det R=\frac{m^2\det M_V}{\mathcal J\left(0\right)^2}
 \det\left(CM_V^{-1}C+B_U^{\mathsf T}B_U/m\right).
\]
Therefore, writing
\[
 R_M=m^2\det M_V\det\left(CM_V^{-1}C+B_U^{\mathsf T}B_U/m\right)=\mathcal J\left(0\right)^2\det R,
\]
Lemma~\ref{lem:ritz-proof} now gives
\[
 \nu_1\nu_2A^2=\frac{\det S}{16d\det R}
 =\frac{\mathcal D\left(0\right)\mathcal J\left(0\right)}{16dR_M}.
\]
Since $\lambda^2=256\pi^4/81$ and $\tau^2-h^2-k^2=3d$, the
geometric-mean target is equivalent to
\[
 \lambda^2\left(\tau^2-h^2-k^2\right)R_M-\mathcal D\left(0\right)\mathcal J\left(0\right)\ge0.
\]

Let $J_*=e_Ae_Be_E^2$.

\subsection*{Normalized matrix formulas}

The following normalization makes the equality point and the invariant
coefficients explicit for both means.

\paragraph{A common normalization for the two means.}
Let
\[
 P=\operatorname{diag}\left(1,1,e_A^{-1/2},e_B^{-1/2},
                               e_E^{-1/2},e_E^{-1/2}\right).
\]
Set
\[
 \widehat M=P^{\mathsf T}MP,\qquad
 \widehat N=\frac1\tau P^{\mathsf T}NP.
\]
Thus the four complementary functions have equilateral energy one,
and the normalized complementary energy block has trace $4$ for every
admissible anisotropy $\left(h/\tau,k/\tau\right)$.
In every normalized formula replace $h,k$ by $h/\tau,k/\tau$;
the leading mass block remains $mI_2$.  Partition $\widehat N$ at
the first two coordinates, writing
\[
 \widehat N=\begin{pmatrix}\widehat A&\widehat B\\
                    \widehat B^{\mathsf T}&\widehat C\end{pmatrix},\qquad
 \widehat\Lambda=-\widehat C^{-1}\widehat B^{\mathsf T},
\]
\[
 \widehat S=\widehat A-\widehat B\widehat C^{-1}\widehat B^{\mathsf T},
 \qquad \widehat R=mI+\widehat\Lambda^{\mathsf T}\widehat M_V\widehat\Lambda,
\]
where $\widehat M_V$ is the complementary mass block. Define
\begin{align*}
 D_0&=\det\widehat N,
 &D_1&=\left.\frac{\mathrm d}{\mathrm dw}
                    \det\left(\widehat N+w\widehat M\right)\right|_{w=0},\\
 J_0&=\det\widehat C,
 &J_1&=\left.\frac{\mathrm d}{\mathrm dw}
                  \det\left(\widehat C+w\widehat M_V\right)\right|_{w=0},\\
 H_D&=D_1J_0-D_0J_1,
 &H_U&=D_0J_0.
\end{align*}
These are precisely the invariant polynomials used for the harmonic
argument.  At $\tau=1$, they are obtained from the unscaled
determinants below by division by $J_*=e_Ae_Be_E^2$ for $D_0,D_1,J_0,J_1$,
and by $J_*^2$ for $H_D,H_U$.

The normalization preserves the trial functions and gives
$\widehat S=S/\tau$ and $\widehat R=R$.
To see how the mixed trace is encoded by determinant derivatives,
consider the Schur complement of the shifted complementary block:
\[
 \widehat S\left(w\right)=\widehat A+wmI_2
 -\widehat B\left(\widehat C+w\widehat M_V\right)^{-1}
             \widehat B^{\mathsf T}.
\]
The mass has no off-diagonal block because $U\perp V$. Differentiating
the inverse matrix therefore gives
\begin{align*}
 \widehat S\left(0\right)&=\widehat S,\\
 \widehat S'\left(0\right)&=mI_2+
 \widehat B\widehat C^{-1}\widehat M_V\widehat C^{-1}
 \widehat B^{\mathsf T}=\widehat R.
\end{align*}
Its determinant is the quotient
\[
 \det\widehat S\left(w\right)=
 \frac{\det\left(\widehat N+w\widehat M\right)}
      {\det\left(\widehat C+w\widehat M_V\right)}.
\]
The determinant derivative formula from
\cite[Chapter~0]{HornJohnson2013} gives
\[
 \left.\frac{\mathrm d}{\mathrm dw}\det\widehat S\left(w\right)\right|_{w=0}
 =\tr\left(\adj\left(\widehat S\right)\widehat R\right).
\]
Evaluating the quotient and its derivative at zero now yields
\[
 \det\widehat S=\frac{D_0}{J_0},\qquad
 \tr\left(\adj\left(\widehat S\right)\widehat R\right)=\frac{H_D}{J_0^2}.
\]
It follows that
\[
 \frac{\tr\left(\adj\left(\widehat S\right)\widehat R\right)}{\det\widehat S}
 =\frac{H_D}{H_U},\qquad
 Y=8\pi^2\left(1-u\right)\frac{H_D}{H_U}.
\]
The normalized pair has physical energy $\tau\widehat S/\left(4d\right)$ and
mass $\widehat R$. Lemma~\ref{lem:ritz-proof} consequently gives
\[
 \frac{H\left(\nu_1,\nu_2\right)L^2}{16\pi^2}
 =\frac{L^2\tau}{32\pi^2d}
 \frac{\det\widehat S}{\tr\left(\adj\left(\widehat S\right)\widehat R\right)}.
\]
The three squared side lengths, up to ordering, are
\[
 \frac{4\tau}{3}\left(1-\eta\cos\left(\theta+2\pi i/3\right)\right),
 \qquad i=0,1,2.
\]
Inserting the definitions of $\tau,h,k$ gives
$4$, $\left(r+s\right)^2$, and $\left(r-s\right)^2$. Define
\[
 S_{\rm geom}=\sum_{i=0}^2
       \sqrt{1-\eta\cos\left(\theta+2\pi i/3\right)}.
\]
Then $L^2=\left(4\tau/3\right)S_{\rm geom}^2$, and the normalized harmonic trial
deficit is exactly $1-S_{\rm geom}^2/Y$. Writing $W=S_{\rm geom}^2$
therefore makes the harmonic trial ratio $W/Y$.

\paragraph{Computing the invariants with two-by-two matrices.}
Eliminate the two scalar directions first. Put
\[
 \mathsf T=\begin{pmatrix}h&k\\k&-h\end{pmatrix},\qquad
 \mathbf w_A=\binom{k}{h},\qquad \mathbf w_B=\binom{h}{-k}.
\]
For the mass-shift variable $\xi$, define
\[
 a_\xi=\tau e_A+\xi m_A,\qquad
 b_\xi=\tau e_B+\xi m_B,
\]
and the following three symmetric two-by-two matrices:
\begin{align*}
 A_\xi^0={}&\left(m\lambda\tau+\xi m\right)I+b_U\mathsf T
       -\frac{g_A^2}{a_\xi}\mathbf w_A\mathbf w_A^{\mathsf T}
       -\frac{g_B^2}{b_\xi}\mathbf w_B\mathbf w_B^{\mathsf T},\\
 C_\xi^0={}&\left(e_E\tau+\xi m_E\right)I+b_E\mathsf T
       -\frac{g_{AE}^2}{a_\xi}\mathbf w_A\mathbf w_A^{\mathsf T}
       -\frac{g_{BE}^2}{b_\xi}\mathbf w_B\mathbf w_B^{\mathsf T},\\
 B_\xi^0={}&g_E\mathsf T
       -\frac{g_Ag_{AE}}{a_\xi}\mathbf w_A\mathbf w_A^{\mathsf T}
       -\frac{g_Bg_{BE}}{b_\xi}\mathbf w_B\mathbf w_B^{\mathsf T}.
\end{align*}
Eliminating the $\phi_A,\phi_B$ coordinates gives
\begin{align*}
 \mathcal J\left(\xi\right)&=a_\xi b_\xi\det C_\xi^0,\\
 \mathcal D\left(\xi\right)&=a_\xi b_\xi\left[
       \det A_\xi^0\det C_\xi^0+\left(\det B_\xi^0\right)^2
       -\tr\left(\adj\left(A_\xi^0\right)B_\xi^0
                      \adj\left(C_\xi^0\right)\left(B_\xi^0\right)^{\mathsf T}\right)\right].
\end{align*}
These formulas compute the shift polynomials $\det\left(C+\xi M_V\right)$ and
$\det\left(N+\xi M\right)$: all apparent denominators cancel. Setting $\tau=1$,
evaluating at $\xi=0$, and differentiating once gives the normalized
$J_0,J_1,D_0,D_1$ after division by $J_*$, with only two-by-two
determinants, adjugates, products and scalar derivatives.

The graph mass is equally explicit.  Write
\begin{align*}
 \Lambda_E&=-\left(C_0^0\right)^{-1}\left(B_0^0\right)^{\mathsf T},\\
 \Lambda_A&=-\frac{\mathbf w_A^{\mathsf T}\left(g_AI+g_{AE}\Lambda_E\right)}{\tau e_A},\\
 \Lambda_B&=-\frac{\mathbf w_B^{\mathsf T}\left(g_BI+g_{BE}\Lambda_E\right)}{\tau e_B}.
\end{align*}
Then
\begin{align*}
 R&=mI+m_E\Lambda_E^{\mathsf T}\Lambda_E
       +m_A\Lambda_A^{\mathsf T}\Lambda_A+m_B\Lambda_B^{\mathsf T}\Lambda_B,\\
 R_M&=\mathcal J\left(0\right)^2\det R.
\end{align*}
This supplies the geometric invariant polynomial using only one
two-by-two inverse.  It is also equal to the Sylvester determinant
expression above, which provides an equivalent expression for the same graph mass.

To compute the common normalized invariant polynomials, use this raw
recipe at $\tau=1$ and then set
\begin{align*}
 D_0&=\frac{\mathcal D\left(0\right)}{J_*},
 &D_1&=\frac{\mathcal D'\left(0\right)}{J_*},\\
 J_0&=\frac{\mathcal J\left(0\right)}{J_*},
 &J_1&=\frac{\mathcal J'\left(0\right)}{J_*},\\
 \widehat R_M&=\frac{\mathcal J\left(0\right)^2\det R}{J_*^2}.
\end{align*}
The geometric polynomial is
\[
 F_G\left(u,v\right)=\frac{\lambda^2\left(1-u\right)\widehat R_M-D_0J_0}
                    {m^2\lambda^2}.
\]
\paragraph{Why the determinant polynomials are invariants.}
For $g\in D_3$, also denote by $g$ its real two-by-two matrix on the
centered coordinates. Let $Q_g$ be the block diagonal matrix of its action on
$\left(u_1,u_2,\phi_A,\phi_B,\phi_1,\phi_2\right)$. Its blocks are the ordinary rotation or
reflection on the two pairs and the scalar transformation signs on
$\phi_A,\phi_B$. The preceding rules imply
\[
 M=Q_g^{\mathsf T}MQ_g,\qquad N\left(\tau,g^{\mathsf T}\mathsf T g\right)=Q_g^{\mathsf T}N\left(\tau,\mathsf T\right)Q_g,
 \quad \mathsf T=\begin{pmatrix}h&k\\k&-h\end{pmatrix}.
\]
Here $N\left(\tau,\mathsf T\right)$ emphasizes the dependence of $N$ on the
metric, with $\tau$ fixed under the group action.
The same relation holds for the complementary blocks and for the
mass-shifted matrices. Determinants, the Schur graph expressions,
and their first mass derivatives are therefore invariant.
The matrix is homogeneous: $N\left(\tau,h,k\right)=\tau N\left(1,h/\tau,k/\tau\right)$.
For coefficient calculations set $\tau=1$, so $h,k$ represent the
normalized components $h/\tau,k/\tau$ of a general shape. Under the
two generators, $h+ik$ is sent to $\omega\left(h+ik\right)$ and
its complex conjugate, respectively. The two real polynomials
\[
 u=h^2+k^2,\qquad v=h^3-3hk^2
\]
are invariant. Their degrees are $2,3$ and their Jacobian determinant
with respect to $h,k$ is
$6k\left(k^2-3h^2\right)$, which is not identically zero. Since the product of
the degrees is $6=\left|D_3\right|$, the invariant-generator criterion of
Shephard and Todd \cite[Theorem 5.1 and \S7]{ShephardTodd1954}
gives $\mathbb R\left[h,k\right]^{D_3}=\mathbb R\left[u,v\right]$: every real polynomial
unchanged by this group action is a polynomial in $u,v$.

These invariants separate the size and direction of the deformation.
At fixed $u=\eta^2$, the quantity $v=\eta^3\cos\left(3\theta\right)$ ranges over
$\left[-u^{3/2},u^{3/2}\right]$. An angular comparison is therefore an inequality
in the single variable $v$. The endpoints correspond to
reflection-symmetric, isosceles shapes. Monotonicity of the explicit
variational targets on this interval will reduce the comparison for
general triangles to one-variable inequalities on an isosceles family.
The eigenvalues themselves need not be differentiated.

For clarity, the small coefficient extraction mentioned in the text
can be written as an explicit $2\times2$ solve. Here $\left[h^n\right]F$ denotes the coefficient of $h^n$ in a polynomial $F$.
The degree assigns weight two to $u$ and weight three to $v$.
For total degree $n\leq10$, at most two monomials $u^pv^q$ have
$2p+3q=n$. If degree $n$ has the
two monomials $u^pv^q,u^rv^s$, put
$\alpha=\left[h^n\right]F$ and $\beta=\left[h^{n-2}k^2\right]F$. Their coefficients are
\[
 c_{pq}=\frac{\beta-\left(r-3s\right)\alpha}{\left(p-3q\right)-\left(r-3s\right)},\qquad
 c_{rs}=\frac{\left(p-3q\right)\alpha-\beta}{\left(p-3q\right)-\left(r-3s\right)}.
\]
The denominator is nonzero. For example at degree six,
$c_{30}=\left(\beta+6\alpha\right)/9$ and $c_{02}=\left(3\alpha-\beta\right)/9$.
Thus at each weighted degree the invariant coefficients follow from
at most two ordinary monomial coefficients.

\section{The geometric mean}\label{sec:geometric-short}

We apply the normalized graph constructed above. Homogeneity gives
$\mathcal D\left(0\right)=\tau^6J_*D_0$,
$\mathcal J\left(0\right)=\tau^4J_*J_0$, and
$R_M=\tau^8J_*^2\widehat R_M$. The target is $m^2\lambda^2J_*^2\tau^{10}F_G\left(u,v\right)$.
We first decrease the target by varying the angle at fixed anisotropy,
then bound the resulting one-variable polynomial at a single endpoint.
The precise comparison is as follows.

\begin{proposition}\label{prop:geometric-comparison}
For $0<u\leq9/10$ and $\left|v\right|\leq u^{3/2}$, one has
$\partial_vF_G\left(u,v\right)<0$. Moreover,
\[
 F_G\left(u,v\right)>\frac{u}{1250}
 \qquad\left(0<u\leq\left(16/17\right)^2,\quad \left|v\right|\leq u^{3/2}\right).
\]
\end{proposition}
\begin{proof}
The invariant description in Section~\ref{sec:trial-space} leaves only the
following thirteen nonconstant coefficients:
\[
 F_G\left(u,v\right)=\sum_{i,j}a_{ij}u^iv^j.
\]
The only nonzero exponent pairs occur in Table~\ref{tab:geometric-coefficients};
all unlisted coefficients are zero. For an entry $n$ in that table,
$n/10^7<a_{ij}<\left(n+1\right)/10^7$.

\begin{table}[ht]
\centering
\caption{Coefficient enclosures for $F_G$: each entry $n$ means
$n/10^7<a_{ij}<\left(n+1\right)/10^7$.}\label{tab:geometric-coefficients}
\begin{tabular}{cr@{\qquad}cr}
\toprule
$\left(i,j\right)$&$n$&$\left(i,j\right)$&$n$\\
\midrule
$\left(1,0\right)$&$2441005$&$\left(0,1\right)$&$-2210992$\\
$\left(2,0\right)$&$-4019925$&$\left(1,1\right)$&$3250111$\\
$\left(3,0\right)$&$1880953$&$\left(0,2\right)$&$75972$\\
$\left(2,1\right)$&$-1075156$&$\left(4,0\right)$&$-309823$\\
$\left(1,2\right)$&$-77829$&$\left(3,1\right)$&$41641$\\
$\left(0,3\right)$&$170$&$\left(5,0\right)$&$-451$\\
$\left(2,2\right)$&$348$&&\\
\bottomrule
\end{tabular}
\end{table}

The constant is exactly zero because at $h=k=0$ the graph is $U$ and its
Ritz values coincide with the equilateral first positive eigenvalue.
Appendix~\ref{app:integrals} gives the exact integration formulas,
the rational constant enclosures, and a worked calculation of $a_{10}$.
All comparisons below use the bounds in Table~\ref{tab:geometric-coefficients}.

Write $F_G=A_G\left(u\right)+B_G\left(u\right)v+C_G\left(u\right)v^2+a_{03}v^3$.  On $0\le u\le9/10$ Table~\ref{tab:geometric-coefficients} gives
\[
 C_G\left(u\right)>\frac1{2000},\qquad a_{03}>0.
\]
For $\Theta_G\left(u\right)=B_G\left(u\right)+2uC_G\left(u\right)+3a_{03}u^2$, its derivative is greater than $1/10$
and $\Theta_G\left(9/10\right)<-1/100$.  These are direct cubic-polynomial estimates:
discard its positive cubic derivative term and use $u\le9/10$ for the
negative linear term.  Thus $\Theta_G\left(u\right)<0$ on this interval.  Because
$\left|v\right|\le u^{3/2}\le u$ and $v^2\le u^3\le u^2$, it follows that
\[
 \partial_vF_G=B_G+2C_Gv+3a_{03}v^2\le \Theta_G\left(u\right)<0.
\]
The worst angle is consequently $v=u^{3/2}$.

Put $p\left(t\right)=F_G\left(t^2,t^3\right)/t^2$.  Its nine coefficients, in increasing
degree, are
\[
 \left(a_{10},a_{01},a_{20},a_{11},a_{30}+a_{02},a_{21},
   a_{40}+a_{12},a_{31}+a_{03},a_{50}+a_{22}\right).
\]
To see that $p'<0$ on $\left[0,1\right]$, let $d_0,\ldots,d_7$ be the coefficients
of $p'$, let $s_j=\sum_{i=0}^jd_i$, and
$t_j=\sum_{i=0}^js_i$.  The same coefficient enclosures give
\[
 s_7<-\frac1{125},\qquad
 t_j<-\frac{11}{50}\quad\left(0\le j\le6\right).
\]
Finite summation by parts is the identity
\[
 \sum_{i=0}^n c_it^i
 =\left(1-t\right)\sum_{j=0}^{n-1}\left(\sum_{i=0}^jc_i\right)t^j
  +\left(\sum_{i=0}^nc_i\right)t^n.
\]
Applying it twice to $p'$ yields
\[
 p'\left(t\right)=\left(1-t\right)^2\sum_{j=0}^5t_jt^j
       +t_6\left(1-t\right)t^6+s_7t^7<0.
\]
Finally, direct substitution into the coefficient bounds gives
\[
 p\left(16/17\right)>\frac1{1250}>0.
\]
Since $\left(16/17\right)^2<9/10$, the angular comparison applies throughout
$0<\eta\leq16/17$. It gives
$F_G\left(u,v\right)\geq F_G\left(u,u^{3/2}\right)=u\,p\left(\sqrt u\right)>u/1250$.
\end{proof}

It remains to relate this simple range to triangles.  For fixed $r$,
$\eta^2=1-3d/\tau^2$ increases with $s^2$.  At $s=2-r$, putting
$x=r-1$ gives
\[
 \eta=\frac{2\left(1-x^2\right)}{2+x^2}.
\]
This decreases for $0\le x\le1$.  Thus every normalized triangle with
$6/5\le r\le2$ has $\eta\le16/17$.  The geometric-mean inequality is
strict unless $\eta=0$, which is precisely the equilateral triangle.
There the graph correction vanishes and the exact low modes give equality.

\paragraph{Quantitative estimate.}
The normalized complementary energy block $\widehat C$ has trace
$4$, so $J_0\leq1$ by the arithmetic--geometric mean inequality.
Its graph Schur complement obeys
$\widehat S\preceq m\lambda I+b_U\widehat{\mathsf T}$, where
$\widehat{\mathsf T}=\tau^{-1}\left(\begin{smallmatrix}h&k\\k&-h\end{smallmatrix}\right)$.
Thus
\[
 D_0J_0=J_0^2\det\widehat S
 \le\det\left(m\lambda I+b_U\widehat{\mathsf T}\right)
 =m^2\lambda^2-b_U^2u\le m^2\lambda^2.
\]
The exact trial deficit for the squared geometric mean is consequently
\[
 \frac{F_G}{F_G+D_0J_0/\left(m^2\lambda^2\right)}
 \ge\frac{F_G}{1+F_G}\ge\frac{u}{1250+u}\ge\frac{u}{1251}.
\]
For triangle side lengths one has
\[
 u=\frac{2\sum_{i<j}\left(\ell_i^2-\ell_j^2\right)^2}
          {\left(\sum_i\ell_i^2\right)^2}\ge2\mathcal A\left(T\right).
\]
Indeed $\ell_i+\ell_j\ge L/2$ and $\sum_i\ell_i^2\le L^2/2$ by the
triangle inequality. Hence the geometric estimate in
Corollary~\ref{cor:stability-proof} follows on this region.

\section{The harmonic mean via angular monotonicity}\label{sec:harmonic-short}
The area comparison depended only on a determinant. The perimeter
comparison also involves the sum of normalized side lengths. We use
the normalized invariant polynomials $D_i,J_i,H_D,H_U,Y$ defined
above, obtained from $\widehat N,\widehat M$.
The variational criterion is $W\leq Y$, where
\[
 W=\left(\sum_{j=0}^2\sqrt{1-\eta\cos\left(\theta+2j\pi/3\right)}\right)^2.
\]
All the denominators and $H_D$ are positive, by the positive graph
mass and energy matrices and the Schur derivative identity.
Put $c=\cos\left(3\theta\right)$, so $v=\eta^3c$, and set $F_H=Y-W$.
The next comparison accounts for the angular dependence of both
the spectral expression and the perimeter.

\begin{proposition}\label{prop:harmonic-comparison}
For each fixed $0<\eta<1$, the function $F_H$ is strictly increasing
in $c\in\left[-1,1\right]$. If $0<\eta\leq\eta_0=1397/1803$, then
$F_H>0$ for every $c\in\left[-1,1\right]$.
\end{proposition}
\begin{proof}
We compare the angular changes of the two sides of $W\leq Y$.
It suffices to show that the angular derivative of $Y$ is larger than that of $W$.
At fixed $\eta$, the identity $v=\eta^3c$ gives
\[
 \left(F_H\right)_c=\eta^3Y_v-W_c.
\]
We will prove $Y_v>-9/32$ and $-W_c\geq\left(9/32\right)\eta^3$.
Their matching constants imply strict positivity of this derivative.
The geometric estimate comes from the elementary symmetric functions
of the three normalized side lengths. For the spectral estimate,
differentiation separates the full determinant contribution from the
complementary determinant contribution: the former is positive, and
the possible loss from the latter is bounded by the same constant.
Combining the two estimates leaves only the isosceles endpoint $c=-1$
to be checked.

\paragraph{The geometric part.}
Put $t_j=\sqrt{1-\eta\cos\left(\theta+2j\pi/3\right)}$,
$S_{\rm geom}=\sum_jt_j$, $Q_{\rm geom}=\sum_{i<j}t_it_j$, and
$p_{\rm geom}=t_0t_1t_2$.
Since $\sum t_j^2=3$, one has $W=S_{\rm geom}^2=3+2Q_{\rm geom}$ and
\[
 p_{\rm geom}^2=1-\frac34u-\frac14v,\qquad
 \left(W-3\right)^2=12-3u+8p_{\rm geom}S_{\rm geom}.
\]
Subscripts $c,v$ below denote partial derivatives with $\eta,u$,
respectively, held fixed. Differentiation at fixed $\eta$ gives
\[
 \left(p_{\rm geom}\right)_c=-\frac{\eta^3}{8p_{\rm geom}},\qquad
 \left(S_{\rm geom}\right)_c=\frac{W_c}{2S_{\rm geom}}.
\]
Differentiating the preceding identity for $\left(W-3\right)^2$,
and using $W-3=2Q_{\rm geom}$, consequently yields
\[
 -W_c=\frac{\eta^3 S_{\rm geom}^2}{4p_{\rm geom}\left(S_{\rm geom}Q_{\rm geom}-p_{\rm geom}\right)}.
\]
The arithmetic--geometric mean inequality gives
\[
 3S_{\rm geom}=\sum_jt_j^3+\sum_{i\ne j}t_i^2t_j\geq9p_{\rm geom}.
\]
Thus $z_*=p_{\rm geom}/S_{\rm geom}\leq1/3$, while $Q_{\rm geom}\leq3$. Consequently
\[
 \frac{p_{\rm geom}\left(S_{\rm geom}Q_{\rm geom}-p_{\rm geom}\right)}{S_{\rm geom}^2}=z_*\left(Q_{\rm geom}-z_*\right)\leq z_*\left(3-z_*\right)\leq\frac89,
 \qquad -W_c\geq\frac9{32}\eta^3.
\]
In particular $W$ decreases with $c$. Its endpoint expressions give
\begin{align}
 W&\leq5-\eta+4\sqrt{1+\eta/2-\eta^2/2}\leq9-u,\label{eq:perimeter-upper}\\
 W&\geq5+\eta+4\sqrt{1-\eta/2-\eta^2/2}\geq9-3u.\label{eq:perimeter-lower}
\end{align}
For the final inequalities, square the nonnegative right sides of the
corresponding root inequalities. After multiplying by $16$, the
differences are respectively
$\eta^2\left(1-\eta\right)^2$ and $3\eta^2\left(1-\eta\right)\left(5+3\eta\right)$.

\paragraph{A structural angular estimate for the harmonic mean.}
Write the common determinant invariants as
\begin{align*}
 D_0&=A_D\left(u\right)+B_D\left(u\right)v+C_Dv^2,
 &D_1&=P_D\left(u\right)+Q_D\left(u\right)v,\\
 N_0&=D_{1,v}D_0-D_1D_{0,v},
 &N_J&=J_{1,v}J_0-J_1J_{0,v}.
\end{align*}
Here $C_D$ is independent of $u$.  The two-by-two elimination formulas
give the particularly useful factorization
\begin{equation}
 C_D=\frac{\left(g_E^2-b_Ub_E\right)\left(g_Ag_{BE}-g_Bg_{AE}\right)^2}
               {e_Ae_Be_E^2}>0.
\label{eq:angular-leading-coeff}
\end{equation}
To verify the factorization, set $\tau=1$ and put
$b=b_U$, $t=b_E$, $x=g_A$, $y=g_B$, $z=g_E$,
$p=g_{AE}$, $q=g_{BE}$.  At degree six the invariant monomials are $u^3$ and $v^2$. The
following two evaluations isolate the coefficient of $v^2$.
Let $\mathcal D^{\left[6\right]}$ denote the homogeneous degree-six part,
in $h,k$, of the unscaled determinant $\mathcal D\left(0\right)$.
At $\left(h,k\right)=\left(1,0\right)$ the value of $\mathcal D^{\left[6\right]}$ is
\[
 \left(-bq^2+2yzq-ty^2\right)\left(bp^2-2xzp+tx^2\right),
\]
whereas at $\left(h,k\right)=\left(0,1\right)$ its value is
$-\left[bpq+txy-z\left(xq+yp\right)\right]^2$.
Subtracting the second value from the first gives
$\left(z^2-bt\right)\left(xq-yp\right)^2$, proving \eqref{eq:angular-leading-coeff}.
For the exact scalar constants specified above, it suffices to note
\begin{equation}
 g_E^2-b_Ub_E>\frac1{500},\qquad
 g_Ag_{BE}-g_Bg_{AE}>\frac1{2000},\qquad
 Q_D\left(u\right)>\frac15\quad\left(0\le u\le1\right).
\label{eq:angular-coarse-scalars}
\end{equation}
For clarity, $Q_D\left(u\right)=q_0+q_1u$, where the scalar formulas give
$q_0>237/1000$ and $q_1>-4/125$, hence $Q_D\left(u\right)>41/200>1/5$.

For each physical pair $\left(u,v\right)$ with $0\le u<1$ and
$\left|v\right|\le u^{3/2}$, the energy and mass matrices are positive definite,
so $D_1=D_0\tr\left(\widehat N^{-1}\widehat M\right)>0$.
As $Q_D>0$, $D_1\left(u,w\right)>0$ also for every $w\in\left[v,u\right]$.
The factorization yields
\[
 \partial_vN_0=-2C_DD_1<0,
 \qquad N_0\left(u,v\right)\ge N_0\left(u,u\right).
\]
Only one quartic is left.  If $x=1-u$ and
$N_0\left(u,u\right)=a_0+a_1x+a_2x^2+a_3x^3+a_4x^4$, substitution in the
same scalar formulas gives
\begin{equation}
 a_0>\frac1{25},\quad a_1>\frac{13}{20},\quad
 a_2+a_3>\frac{13}{20},\quad a_3<0,\quad a_4>0.
\label{eq:angular-shift-bound}
\end{equation}
Consequently, since $0\le x\le1$,
\begin{equation}
 N_0\left(u,v\right)\ge N_0\left(u,u\right)
 >\frac1{25}+\frac{13}{20}\left(1-u\right)
                     +\frac{13}{20}\left(1-u\right)^2>0.
\label{eq:n0-positive}
\end{equation}

The complementary determinant contributes a negative term to $Y_v$.
Its invariant coefficients give
\begin{equation}
 J_0\left(u,v\right)\ge 1-\frac{411}{500}u+\frac{17}{400}u^2,
 \qquad
 N_J\left(u,v\right)\le\frac{2727-758u}{10^6}.
\label{eq:complementary-bounds}
\end{equation}
Indeed, $J_{0,v}<0$, $\left|v\right|\le u$, and $N_J$ is a quadratic in $u$
whose quadratic coefficient is negative.  Write
$a=411/500$, $b=17/400$.  As $1-au\ge1-a>0$,
\[
 J_0^2\ge\left(1-au\right)^2+2b\left(1-a\right)u^2.
\]
Using $8\pi^2<79$ and $\left(2727-758u\right)/10^6>0$, we obtain
\begin{align*}
 \frac9{32}J_0^2-8\pi^2\left(1-u\right)N_J
 &\ge\frac9{32}\left[\left(1-au\right)^2+2b\left(1-a\right)u^2\right]
       -\frac{79}{10^6}\left(1-u\right)\left(2727-758u\right)\\
 &>\frac{658-1871u+1344u^2}{10000}>0.
\end{align*}
The last quadratic is positive because its leading coefficient is
positive and its discriminant is
$-36767/10^8$ after the common normalization.
Combining this with \eqref{eq:n0-positive} in
\[
 Y_v=8\pi^2\left(1-u\right)
       \left(\frac{N_0}{D_0^2}-\frac{N_J}{J_0^2}\right)
\]
proves $Y_v>-9/32$.  Together with the geometric inequality
$-\partial_c S_{\rm geom}^2\ge\left(9/32\right)\eta^3$, this proves
\[
 \left(F_H\right)_c>0\quad\left(0<\eta<1\right).
\]

The bounds \eqref{eq:angular-coarse-scalars}, \eqref{eq:angular-shift-bound}, and
\eqref{eq:complementary-bounds} follow by substitution of the fifteen explicit scalar
constants in the two-by-two formulas. Table~\ref{tab:harmonic-angular-coefficients}
in Appendix~\ref{app:integrals} records the coefficients of $Q_D$, $J_0$,
$N_J$, and the shifted quartic $N_0\left(u,u\right)$ that give these bounds.
The rational
enclosures
$3.14159265358979323846<\pi<3.14159265358979323847$ and
$1.73205080756887729352<\sqrt3<1.73205080756887729353$
suffice; the margins in the displayed bounds are much larger than
the resulting interval widths.

\paragraph{The isosceles boundary.}
Define the degree-ten polynomial $p_H$ by the exact identity
\begin{equation}
 8\pi^2\left(1-t^2\right)H_D\left(t^2,-t^3\right)
 -\left(9-t^2\right)H_U\left(t^2,-t^3\right)=t^2p_H\left(t\right).                \label{eq:iso-target}
\end{equation}
Its coefficient vector, in ascending powers of $t$, differs by less
than $10^{-6}$ in every entry from
\begin{equation}
\begin{array}{rrrrrr}
11.360043&6.718733&-35.527940&-18.530405&35.266653&13.807750\\
-13.711071&-1.881133&2.377526&-.314987&-.102749&
\end{array}                                                    \label{eq:iso-coefficients}
\end{equation}
Let $t_*=31/40>\eta_0$. These bounds imply, for $0\leq t\leq t_*$,
\begin{align*}
 p_H'\left(t\right)&\leq6.719-71.055t-55.591t^2+141.067t^3
               +69.039t^4-82.266t^5+19.021t^7\\
 &\leq6.719-16t-8.8t^2.
\end{align*}
For the last step, the function $141.067x-82.266x^2$ increases on
$0\leq x\leq t_*^2$, and
\[
 141.067t_*^2-82.266t_*^4<55.050876,\qquad
 69.039t_*^2+19.021t_*^5<46.784471.
\]
Thus $p_H$ decreases on $\left[2/5,t_*\right]$. On $\left[0,2/5\right]$, discarding
positive terms in \eqref{eq:iso-coefficients} gives $p_H>4.43$. Substitution at the single
rational endpoint gives
\[
 p_H\left(31/40\right)>.16819>\frac16.
\]
Hence $p_H>1/6$ throughout $\left[0,t_*\right]$. Equations \eqref{eq:perimeter-upper} and \eqref{eq:iso-target} prove
$F_H>0$ at $c=-1$ for $0<\eta\leq\eta_0$.

Since $\left(F_H\right)_c>0$, the same conclusion holds for every
$c\in\left[-1,1\right]$ at fixed $0<\eta\leq\eta_0=1397/1803$.
\end{proof}

Thus the symmetric trial pair proves the harmonic inequality on
the whole low-$\eta$ region. Lemma~\ref{lem:high-horizontal}
already covers every remaining shape.

\paragraph{A matrix bound for the harmonic stability denominator.}
The explicit scalar constants satisfy
\[
 \max\left\{\frac{m_A}{e_A},\frac{m_B}{e_B},
                         \frac{m_E}{e_E}\right\}<\frac1{100},
 \qquad
 \frac{g_A^2}{e_A}+\frac{g_B^2}{e_B}
                        +\frac{2g_E^2}{e_E}<4.
\]
Thus $\widehat M_V\preceq I/100$ and
$\left\|\widehat B\right\|_F^2\le4u$.
The gradient metric has eigenvalues $1\pm\eta$, so
$\widehat C\succeq\left(1-\eta\right)I$.  For the graph map
$\widehat\Lambda=-\widehat C^{-1}\widehat B^{\mathsf T}$ and its mass matrix
$\widehat R=mI+\widehat\Lambda^{\mathsf T}\widehat M_V\widehat\Lambda$, it follows that
\[
 \tr\widehat R
 \le\frac34+\frac{u}{25\left(1-\eta\right)^2}
 \le\frac34+\frac1{25}\left(\frac{31}{9}\right)^2
 <\frac54
 \qquad\left(0\le\eta\le31/40\right).
\]
Also, $0\prec\widehat S\preceq\widehat A$ and
\[
 \left\|\widehat A\right\|_{\rm op}
 =m\lambda+b_U\eta
 <\frac{20}{3}+\frac{2511}{640}<11.
\]
For a positive two-by-two matrix, the eigenvalues of its adjugate
are its eigenvalues in the opposite order.  Hence
$\adj\widehat S\preceq11I$.
The normalized complementary block has trace $4$, so $J_0\le1$.
The exact derivative identity now gives directly
\[
 H_D=J_0^2\tr\left(\adj\left(\widehat S\right)\widehat R\right)
 <11\cdot\frac54<14<20.
\]

\paragraph{Quantitative estimate.}
For the side asymmetry defined above, \eqref{eq:perimeter-lower} gives
\[
 \mathcal A\left(T\right)=\frac{9-W}{W}
 \leq\frac{3u}{9-3u}\leq\frac u2.
\]
Write $\delta_H\left(T\right)=1-H\left(\mu_2,\mu_3\right)L^2/\left(16\pi^2\right)$ for the
actual spectral deficit. The Ritz principle bounds it below by the
deficit of either trial pair used here.
On $0<\eta\leq\eta_0$, $F_H>0$ and the trial deficit
$\delta=F_H/\left(F_H+W\right)=1-W/Y$ increases with $c$, since
$\left(F_H\right)_c>0$ and $W_c<0$. At its isosceles minimum, \eqref{eq:iso-target} and
the matrix bound $H_D<20$ yield
\[
 \delta\geq\frac{u\,p_H\left(\sqrt u\right)}{8\pi^2\left(1-u\right)H_D}
 >\frac{u}{9600}\geq\frac{\mathcal A\left(T\right)}{4800}.
\]
The same lower bound therefore holds for $\delta_H\left(T\right)$.
Here $u$ remains fixed when the angle varies, so the last comparison
controls the asymmetry of the original triangle. On $\eta\geq\eta_0$,
Lemma~\ref{lem:high-horizontal} gives the uniform deficit
\[
 \delta_H\left(T\right)>\frac{661}{1577536}>\frac1{9600}
 \geq\frac{\mathcal A\left(T\right)}{4800}.
\]
This proves the harmonic estimate in Corollary~\ref{cor:stability-proof}. At $\eta=0$, the deficit and asymmetry both vanish.

\section{Completion of the proof}\label{sec:completion}

The two scalar comparisons and the horizontal estimates cover the
entire normalized shape domain. We now combine them and identify
the equality cases.

\begin{proof}[Proof of Theorem~\ref{thm:main-algebraic}]
For the harmonic mean, Proposition~\ref{prop:harmonic-comparison} covers
$0<\eta\leq\eta_0$, and the uniform horizontal cap covers
$\eta\geq\eta_0$.
For the geometric mean, the horizontal product estimate covers
$1<r\leq6/5$, and Proposition~\ref{prop:geometric-comparison} applies
to $6/5\leq r\leq2$ by the anisotropy bound proved in
Section~\ref{sec:geometric-short}.
Both inequalities are strict when $\eta>0$, hence for every
nonequilateral triangle. At $\left(r,s\right)=\left(2,0\right)$, the graph correction
vanishes and the exact equilateral modes give
$\mu_2=\mu_3=4\pi^2/9$, $L=6$, and $A=\sqrt3$.
Both equalities follow. Similarity invariance completes the proof.
\end{proof}

\begin{proof}[Proof of Corollary~\ref{cor:stability-proof}]
The harmonic stability estimate on both invariant regions was proved
in Section~\ref{sec:harmonic-short}. In the geometric cap, the deficit
exceeds $7/40>1/1251\geq2\mathcal A\left(T\right)/1251$, since
$\mathcal A\left(T\right)\leq1/2$. Outside the cap, the estimate
$2\mathcal A\left(T\right)/1251$ was proved in
Section~\ref{sec:geometric-short}. Monotonicity of the means and
the Ritz principle transfer all trial deficit bounds to the actual
eigenvalues. At the equilateral triangle both deficits and asymmetry vanish.
\end{proof}

\appendix
\section{Elementary scalar integration and rational bounds}
\label{app:integrals}

The fifteen scalar integrals in the symmetric space can be evaluated
by the following elementary formulas. We record the moment formula, the trigonometric integration recursion,
and representative calculations of a scalar integral and an invariant
coefficient. These formulas specify all arithmetic inputs used in the proof.
Every terminating decimal in the paper denotes an exact rational number.

\paragraph{Polynomial moments.}
For nonnegative integers $p,q$,
\[
 \int_Dx^pz^q\dd x\dd z=
 \begin{cases}
 \displaystyle\frac{p!q!}{2^{p+q+1}\left(p+q+2\right)!},&p\ \text{even},\\
 0,&p\ \text{odd}.
 \end{cases}
\]
The corresponding horizontal moments on $T$ follow from barycentric
coordinates, the coefficients $\lambda_0,\lambda_1,\lambda_2\geq0$
with sum one in the expression of a point as a convex combination
of its three vertices. Successive integration over
$\lambda_1,\lambda_2\geq0$, $\lambda_1+\lambda_2\leq1$, gives
\[
 \mathbb E\left(\lambda_0^p\lambda_1^q\lambda_2^j\right)
 =\frac{2p!q!j!}{\left(p+q+j+2\right)!}.
\]
Since $X_1=-\lambda_0+\lambda_1+a\lambda_2$, the multinomial
expansion yields
\[
 \mathbb E X_1^k=\frac{2}{\left(k+1\right)\left(k+2\right)}
 \sum_{\ell=0}^k a^\ell\sum_{p=0}^{k-\ell}\left(-1\right)^p
 =\frac{2}{\left(k+1\right)\left(k+2\right)}
 \sum_{\substack{0\leq\ell\leq k\\k-\ell\ \mathrm{even}}}a^\ell.
\]
\paragraph{Mixed polynomial and trigonometric integrals.}
Expand the sines and cosines in the displayed symmetric reference
functions into exponentials.
The remaining integrals take the form
\[
 \int_Dx^pz^q e^{i\pi\left(\alpha x+\beta z\right)}\dd x\dd z,
\]
where $\alpha,\beta$ are real frequencies and $i^2=-1$.
For a nonnegative integer $n$ and real $c$ set
$K_n\left(c\right)=\int_0^{1/2}z^ne^{i\pi cz}\dd z$.
Then
\[
 K_n\left(0\right)=\frac{2^{-n-1}}{n+1},\qquad
 K_0\left(c\right)=\frac{e^{i\pi c/2}-1}{i\pi c}\quad\left(c\ne0\right),
\]
and, for $c\ne0$ and $n\geq1$,
\[
 K_n\left(c\right)=
 \frac{2^{-n}e^{i\pi c/2}-nK_{n-1}\left(c\right)}{i\pi c}.
\]
Define
\[
 \mathscr W_{n,q}\left(c\right)=
 \sum_{j=0}^n\binom nj2^{j-n}\left(-1\right)^j K_{q+j}\left(c\right).
\]
For $\alpha\ne0$, integrating in $x$ gives
\begin{align*}
 &\int_Dx^pz^qe^{i\pi\left(\alpha x+\beta z\right)}\dd x\dd z\\
 &\quad=\sum_{j=0}^p
 \frac{\left(-1\right)^jp!}{\left(p-j\right)!\left(i\pi \alpha\right)^{j+1}}
 \left[e^{i\pi \alpha/2}\mathscr W_{p-j,q}\left(\beta-\alpha\right)
       -\left(-1\right)^{p-j}e^{-i\pi \alpha/2}\mathscr W_{p-j,q}\left(\beta+\alpha\right)\right].
\end{align*}
For $\alpha=0$ the integral is zero for odd $p$ and is
$2\mathscr W_{p+1,q}\left(\beta\right)/\left(p+1\right)$ for even $p$.
All endpoint phases needed here are sixth roots of unity.
Conjugate terms combine to real numbers in
$\mathbb Q\left[\pi,\pi^{-1},\sqrt3\right]$, the ring of finite polynomial
expressions in these constants with rational coefficients.

\paragraph{A worked polynomial scalar.}
For $\phi_B=x_c\left(x_c^2-3y_c^2\right)$, differentiation in the equilateral coordinates
gives $\phi_{B,x_c}=3\left(x_c^2-y_c^2\right)$ and $\phi_{B,y_c}=-6x_cy_c$. Since
$\partial_z=\sqrt3\partial_{y_c}$,
\[
 e_B=\int_D\left(\phi_{B,x_c}^2+\phi_{B,y_c}^2\right)\dd x\dd z=9\int_D\left(x_c^2+y_c^2\right)^2\dd x\dd z.
\]
The monomial formula in this appendix gives
\[
 \int_Dx_c^4\dd x\dd z=\int_Dy_c^4\dd x\dd z=\frac1{960},\qquad
 \int_Dx_c^2y_c^2\dd x\dd z=\frac1{2880},
\]
and hence $e_B=9\left(2/960+2/2880\right)=1/40$.
It also gives
\[
 \int_Dx_c^6\dd x\dd z=\frac1{7168},\quad
 \int_Dx_c^4y_c^2\dd x\dd z=\frac{13}{322560},\quad
 \int_Dx_c^2y_c^4\dd x\dd z=\frac{19}{967680},
\]
so
\[
 m_B=\frac1{7168}-6\frac{13}{322560}
                         +9\frac{19}{967680}=\frac1{13440}.
\]
The polynomial, mixed polynomial--trigonometric, and projection
scalars listed above are evaluated by the same monomial formula and
the exponential recursion already given in the appendix.

\paragraph{Leading coefficients of the variational bounds.}
Write $e=m\lambda=2\pi^2/3$ and $\beta=b_U/e$. Define the two positive scalar combinations
\[
 \kappa_E=\frac{g_A^2}{e_A}+\frac{g_B^2}{e_B}
                        +\frac{2g_E^2}{e_E},\qquad
 \kappa_M=\frac{m_Ag_A^2}{e_A^2}+\frac{m_Bg_B^2}{e_B^2}
                        +\frac{2m_Eg_E^2}{e_E^2}.
\]
Here $O\left(\eta^j\right)$ denotes a remainder bounded by a constant times
$\eta^j$ as $\eta\to0$, uniformly in the angle.
At $\tau=1$, the graph matrix formulas give
\[
 \det\widehat S=e^2-\left(b_U^2+e\kappa_E\right)u+O\left(\eta^3\right),\quad
 \det\widehat R=m^2+m\kappa_Mu+O\left(\eta^3\right),\quad
 J_0^2=1+O\left(u\right).
\]
The constant in the geometric target is zero, so multiplication by
$J_0^2$ does not affect the coefficient of $u$. Thus the first
coefficient in its table is the explicit scalar expression
\[
 a_{10}=-1+\beta^2+\frac{\kappa_E}{e}+\frac{\kappa_M}{m}.
\]
Substitution of the rational intervals for $\pi,\sqrt3$ yields
\begin{align*}
 3.69243283196870&<\kappa_E<3.69243283196871,\\
 .03409865452445417&<\kappa_M<.03409865452445418.
\end{align*}
\[
 .24410051071219619<a_{10}<.24410051071219621,
\]
which in particular verifies the first printed table enclosure
$2441005/10^7<a_{10}<2441006/10^7$ directly.
The same expansion of the harmonic trace gives
\[
 Y=9+\frac92\left(\frac{\kappa_E}{e}+2\beta^2
                         +\frac{\kappa_M}{m}-2\right)u
          +O\left(\eta^3\right).
\]
Since $W=9-\left(9/8\right)u-\left(9/32\right)v+O\left(\eta^4\right)$, their difference has the expansion
\[
 Y-W=\gamma_Hu+O\left(\eta^3\right),\qquad
 \gamma_H=\frac98+\frac92\left(\frac{\kappa_E}{e}+2\beta^2
                    +\frac{\kappa_M}{m}-2\right)>\frac38.
\]
The affine space $U$ alone has
$Y_U=9\left(1-u\right)/\left(1-\beta^2u\right)$, whose difference from $W$ has
quadratic term $\left[9\left(\beta^2-1\right)+9/8\right]u<0$ for $u>0$.
These expansions identify the quadratic obstruction and the positive
gain of the complementary directions. The angular and endpoint estimates
in the body of the proof establish the inequalities on the full shape domain.

\paragraph{Rational bounds for the constants.}
All printed scalar and coefficient bounds follow with the enclosures
\begin{align*}
 \pi_0=3.14159265358979323846&<\pi<\pi_0+10^{-20},\\
 s_0=1.73205080756887729352&<\sqrt3<s_0+10^{-20}.
\end{align*}
For an entirely rational verification, for integers $n\geq1$ and
real $q>1$ put
\[
 A_n\left(q\right)=\sum_{j=0}^{n-1}\frac{\left(-1\right)^j}{\left(2j+1\right)q^{2j+1}}.
\]
Machin's tangent-addition identity
\cite[Section~2.1]{BorweinBorweinGalway2004}, together with the
alternating-series bounds, gives
\[
 \pi_0<16A_{16}\left(5\right)-4A_{17}\left(239\right)<\pi
 <16A_{17}\left(5\right)-4A_{16}\left(239\right)<\pi_0+10^{-20}.
\]
Also $s_0^2<3<\left(s_0+10^{-20}\right)^2$. In particular $157/50<\pi$,
$\pi^2<10$, $8\pi^2<79$, and $\sqrt3<7/4$.

For precision, a rational interval $\left[a,b\right]$ means the set of real
numbers $x$ with $a\leq x\leq b$. Addition uses
$\left[a,b\right]+\left[c,d\right]=\left[a+c,b+d\right]$; multiplication uses the smallest and largest
of $ac,ad,bc,bd$; and division is multiplication by the reciprocal
interval when the divisor does not contain zero. These rules apply to
the stated scalar formulas and the two-by-two determinant identities.
The monomial coefficients are extracted before substituting a shape.
They yield the bounds of Table~\ref{tab:geometric-coefficients},
\eqref{eq:iso-coefficients}, and the angular scalar comparisons. Positive square roots in the normalization can
be enclosed by squaring rational endpoints; equivalently, compute
the unscaled two-by-two expressions at $\tau=1$ and divide by
$J_*$ or $J_*^2$, so that no extra square roots occur. Apparent
matrix denominators cancel in the determinant polynomials.
Thus every comparison is a finite rational inequality on explicitly
given constants, and the continuous shape comparisons are proved in the preceding sections.

\paragraph{Coefficients for the harmonic angular comparison.}
Write
\[
 Q_D\left(u\right)=q_0+q_1u,\qquad
 J_0\left(u,v\right)=1+j_{10}u+j_{20}u^2+j_{01}v,\qquad
 N_J\left(u,v\right)=n_0+n_1u+n_2u^2,
\]
and retain the coefficients $a_0,\ldots,a_4$ defined by
$N_0\left(u,u\right)=\sum_{i=0}^4a_i\left(1-u\right)^i$.
Expanding the two-by-two determinant formulas at $\tau=1$ and
extracting the invariant coefficients gives
Table~\ref{tab:harmonic-angular-coefficients}. Its rational enclosures
use the bounds for $\pi$ and $\sqrt3$ above. The stated constant terms
and vanishing coefficients follow from the same determinant formulas.

\begin{table}[ht]
\centering
\caption{Harmonic angular coefficients: each entry $n$ bounds the
corresponding coefficient strictly between $n/10^7$ and $\left(n+1\right)/10^7$.}
\label{tab:harmonic-angular-coefficients}
\begin{tabular}{cr@{\qquad}cr}
\toprule
Coefficient&$n$&Coefficient&$n$\\
\midrule
$q_0$&$2370139$&$q_1$&$-314556$\\
$j_{10}$&$-7267303$&$j_{20}$&$425915$\\
$j_{01}$&$-947346$&$n_0$&$27266$\\
$n_1$&$-7585$&$n_2$&$-95$\\
$a_0$&$406395$&$a_1$&$6540211$\\
$a_2$&$7325337$&$a_3$&$-767150$\\
$a_4$&$761285$&&\\
\bottomrule
\end{tabular}
\end{table}

For example, the complementary determinant bounds follow from
\[
 j_{10}+j_{01}>-.8214649>-.822,\qquad
 j_{20}>.0425915>.0425,\qquad j_{01}<0,
\]
\[
 n_0<.0027267<.002727,\qquad
 n_1<-.0007584<-.000758,\qquad n_2<0.
\]
Indeed $v\leq u^{3/2}\leq u$ when $0\leq u\leq1$.
Likewise, $a_2+a_3>.6558187>.65$, and the other shifted-quartic
bounds in \eqref{eq:angular-shift-bound} follow directly from the table.


\begin{thebibliography}{99}

\bibitem{OW2009}
M. Ashbaugh, R. Benguria, R. Laugesen and T. Weidl (organizers),
\emph{Low eigenvalues of Laplace and Schr\"odinger operators},
Oberwolfach Rep. \textbf{6} (2009), no.~1, 355--428.
\bibitem{BogoselHenrotMichetti2024}
B. Bogosel, A. Henrot and M. Michetti,
\emph{Optimization of Neumann eigenvalues under convexity and geometric constraints},
SIAM J. Math. Anal. \textbf{56} (2024), no.~6, 7327--7349.
\bibitem{BorweinBorweinGalway2004}
J. M. Borwein, D. Borwein and W. F. Galway,
\emph{Finding and excluding $b$-ary Machin-type individual digit formulae},
Canad. J. Math. \textbf{56} (2004), no.~5, 897--925.
\bibitem{EP2013}
C. Enache and G. A. Philippin,
\emph{Some inequalities involving eigenvalues of the Neumann Laplacian},
Math. Methods Appl. Sci. \textbf{36} (2013), no.~16, 2145--2153.
\bibitem{EndoOsting2026}
R. Endo and B. Osting,
\emph{Maximizing the fundamental Laplace--Neumann eigenvalue on quadrilaterals},
preprint (2026), \href{https://arxiv.org/abs/2609.02784}{arXiv:2609.02784}.

\bibitem{HenrotLemenantLucardesi2024}
A. Henrot, A. Lemenant and I. Lucardesi,
\emph{An isoperimetric problem with two distinct solutions},
Trans. Amer. Math. Soc. \textbf{377} (2024), no.~4, 2367--2411.
\bibitem{HornJohnson2013}
R. A. Horn and C. R. Johnson,
\emph{Matrix Analysis}, second edition, Cambridge University Press, Cambridge, 2013.
\bibitem{LS2009}
R. S. Laugesen and B. A. Siudeja,
\emph{Maximizing Neumann fundamental tones of triangles},
J. Math. Phys. \textbf{50} (2009), no.~11, 112903, 18~pp.
\bibitem{LS2011}
R. S. Laugesen and B. A. Siudeja,
\emph{Sums of Laplace eigenvalues---rotationally symmetric maximizers in the plane},
J. Funct. Anal. \textbf{260} (2011), no.~6, 1795--1823.
\bibitem{LS2017}
R. S. Laugesen and B. A. Siudeja,
\emph{Triangles and other special domains}, in \emph{Shape Optimization and Spectral Theory},
A. Henrot (ed.), De Gruyter Open, Warsaw, 2017, 149--200.
\bibitem{LenaRohleder2024}
C. L\'ena and J. Rohleder,
\emph{Estimates for the lowest Neumann eigenvalues of parallelograms and domains of constant width},
Anal. Math. Phys. \textbf{14} (2024), article~42, 24~pp.
\bibitem{McCartin2002}
B. J. McCartin,
\emph{Eigenstructure of the equilateral triangle. Part II: The Neumann problem},
Math. Probl. Eng. \textbf{8} (2002), no.~6, 517--539.
\bibitem{ShephardTodd1954}
G. C. Shephard and J. A. Todd,
\emph{Finite unitary reflection groups}, Canad. J. Math. \textbf{6} (1954), 274--304.
\bibitem{Szego1954}
G. Szeg\H{o}, \emph{Inequalities for certain eigenvalues of a membrane of given area},
J. Rational Mech. Anal. \textbf{3} (1954), 343--356.
\bibitem{Weinberger1956}
H. F. Weinberger, \emph{An isoperimetric inequality for the $N$-dimensional free membrane problem},
J. Rational Mech. Anal. \textbf{5} (1956), 633--636.
\end{thebibliography}
\end{document}